\documentclass[11pt]{article}

\usepackage{amsmath, amssymb, amsthm} 

\usepackage{hyperref}

\usepackage{dsfont}

\usepackage{tikz}
\usetikzlibrary{decorations,decorations.markings,decorations.pathreplacing,decorations.pathmorphing}
\usetikzlibrary{positioning,chains,fit,shapes,calc}

\allowdisplaybreaks
\expandafter\let\expandafter\oldproof\csname\string\proof\endcsname
\let\oldendproof\endproof
\renewenvironment{proof}[1][\proofname]{%
	\oldproof[\bf #1]%
}{\oldendproof}

\allowdisplaybreaks

\theoremstyle{plain}
\newtheorem{theorem}{Theorem}[section]
\newtheorem{lemma}[theorem]{Lemma}
\newtheorem{claim}[theorem]{Claim}
\newtheorem{proposition}[theorem]{Proposition}

\newtheorem{conjecture}[theorem]{Conjecture}

\newtheorem{definition}[theorem]{Definition}

\newtheorem{fact}[theorem]{Fact}

\RequirePackage[normalem]{ulem} 
\RequirePackage{color}\definecolor{RED}{rgb}{1,0,0}\definecolor{BLUE}{rgb}{0,0,1} 

\title{\vspace{-0.9cm} Two Erd\H{o}s--Hajnal-type theorems for forbidden order-size pairs}
\author{
Fabian Arnold\thanks{Department of Mathematics, ETH, R\"amistrasse 101, Z\"urich 8092, Switzerland.  Email: faarnold@student.ethz.ch, benjamin.sudakov@math.ethz.ch. 
Research supported in part by SNSF grant 200021-228014.} \and
Lior Gishboliner\thanks{Department of Mathematics, University of Toronto, Canada. Research supported by an NSERC Discovery grant.}
\and Benny Sudakov\footnotemark[1]
}
\date{}

\begin{document}

\maketitle

\begin{abstract}
The celebrated Erd\H{o}s-Hajnal conjecture states that any graph without a fixed induced subgraph $H$ contains a very large homogeneous set. A direct analog of this conjecture is not true for hypergraphs. In this paper we present two natural variants of this problem which do hold for hypergraphs. We show that for every $r \geq 3$, $m \geq m_0(r)$ and $0 \leq f \leq \binom{m}{r}$, if an $r$-graph $G$ does not contain $m$ vertices spanning exactly $f$ edges, then $G$ contains much larger homogeneous sets than what is guaranteed to exist in general $r$-graphs. We also prove that if a $3$-graph $G$ does not contain homogeneous sets of polynomial size, then for every $m \geq 3$, there are $\Omega(m^3)$ values of $f$ such that $G$ contains $m$ vertices spanning exactly $f$ edges. This makes progress on a problem raised by Axenovich, Brada\v{c}, Gishboliner, Mubayi and Weber.  
\end{abstract}

\section{Introduction}

Ramsey's theorem states that for every $r \geq 2$ and $t \geq r$, there is a (smallest) integer $R = R_r(t)$ such that every $r$-uniform hypergraph on $R$ vertices contains a homogeneous set (i.e., a clique or an independent set) of size $t$. 
The quantitative dependence of $R$ on $t$ and $r$ is a central question in Ramsey theory. 
The seminal theorems of Erd\H{o}s-Szekeres \cite{ES} and Erd\H{o}s \cite{Erdos} imply that for graphs, i.e., for $r=2$, the function $R_2(t)$ grows exponentially in $t$. Inverting the relation, this means that every $n$-vertex graph has a homogeneous set of size $\Omega(\log n)$, and this is best possible. For $r$-graphs with $r \geq 3$, Erd\H{o}s and Rado \cite{ER} proved that every $n$-vertex $r$-graph contains a homogeneous set of size $\Omega(\log_{(r-1)}(n))$, where $\log_{(k)}(x)$ is the $k$th iterated logarithm, i.e., $\log_{(1)}(x) = \log x$ and $\log_{(k)}(x) = \log(\log_{(k-1)}(x))$. 
It is conjectured \cite{EHR} that this is best possible, i.e., that there exist $n$-vertex $r$-graphs with no homogeneous sets of size $(\log_{(r-1)}(n))^{O(1)}$, 
but this remains open. Erd\H{o}s and Hajnal proved the famous {\em stepping-up lemma} (see, e.g., \cite{Ramsey_book}), which implies that $R_r(t) \geq 2^{R_{r-1}(\Omega(t))}$ for $r \geq 4$ (see also \cite{CFS_steppingup} for an improvement). Hence, to prove the aforementioned conjecture, it suffices to show that there exist $n$-vertex $3$-graphs with no homogeneous sets of size 
$(\log\log (n))^{O(1)}$ (since then the stepping up lemma would imply the conjecture for all $r \geq 3$). 
The best known upper bound is only $(\log n)^{\Theta(1)}$, which translates into a gap of one exponential between the best known upper and lower bounds on $R_r(t)$ (e.g., between exponential and double-exponential for $r=3$).
In fact, the stepping up lemma, combined with the stepping down argument (stated here as Theorem \ref{thm:ER_basic}), show that the stepping up construction gives a tight bound (in terms of the number of exponents) on $R_r(t)$ for all $r \geq 4$; but this bound is expressed in terms of $R_3(t)$ and is hence unknown (since the best known lower and upper bounds on $R_3(t)$ differ in one exponent). 
We note that the analogous problem for 4 colors is settled; it is known that there exist 4-colorings of the edges of the complete $3$-graph $K_n^{(3)}$ with no monochromatic set of size $C\log\log n$. 

Erd\H{o}s and Hajnal \cite{EH_Ramsey} made the famous conjecture that if a (2-uniform) graph $G$ is induced $H$-free (for any given fixed graph $H$) then $G$ must have much larger homogeneous sets than what is guaranteed by Ramsey's theorem. More precisely, they conjectured that for every graph $H$ there is $c_H > 0$ such that every $n$-vertex induced $H$-free graph has a homogeneous set of size $n^{c_H}$. This problem has received a lot of attention, with some substantial recent progress; see, e.g., \cite{BNSS,NSS_VC,nguyen2023} and the references therein. Still, the general conjecture remains wide open.

The direct analogue of the Erd\H{o}s-Hajnal conjecture for $r$-graphs is that for every $r$-graph $H$ there is $c_H > 0$ such that every $n$-vertex induced $H$-free $r$-graph contains a homogeneous set of size $(\log_{(r-2)}(n))^{c_H}$. In other words, the number of iterated logarithms is one less than what is conjectured to hold for general $r$-graphs. However, there is strong evidence that this analogous conjecture is false for all $r \geq 4$. Indeed, as mentioned above, the stepping up construction gives a tight bound (in terms of the number of exponents) on $R_r(t)$ for all $r \geq 4$. 
Conlon, Fox and Sudakov \cite{CFS_hyper} showed that the stepping up construction is induced $H$-free for certain $r$-graphs $H$. 
Hence,
induced $H$-free $r$-graphs can have essentially the same size of homogeneous sets as general $r$-graphs (for certain $H$). In particular, if the conjecture regarding the growth of $R_r(t)$ is correct, meaning that there exist $n$-vertex $r$-graphs with no homogeneous sets of size $(\log_{(r-1)}(n))^{O(1)}$,  
then the aforementioned hypergraph analogue of the Erd\H{o}s-Hajnal conjecture fails.

In this paper we consider a version of the Erd\H{o}s-Hajnal conjecture which does in fact hold for $r$-graphs for every $r$. 
We need the following definition: An {\em $(m,f)$-subset} of a hypergraph is a set of $m$ vertices which induces exactly $f$ edges. 
Namely, instead of forbidding a specific induced subgraph $H$, we forbid an {\em order-size pair} $(m,f)$. 
Erd\H{o}s-Hajnal-type problems for forbidden order-size pairs have attracted interest recently \cite{MS_OrderSizePair,GT,ABGMW}. 
In \cite{ABGMW} it was shown that for every $m$ and $0 \leq f \leq \binom{m}{2}$, a (2-uniform) graph with no $(m,f)$-subset contains a homogeneous set of size $\Omega(n^{\frac{1}{m-1}})$. Our first theorem proves such a result for hypergraphs of all uniformities.

\begin{theorem} \label{main-result-1}
    For every $r \geq 3$ there is a constant $a(r)>0$ such that for every $m \geq 5 r^2$ and $0 \leq f \leq \binom{m}{r}$, and for all large enough $n$, every $n$-vertex $r$-graph with no $(m,f)$-subset contains a homogeneous set of size at least
    $\big(\log_{(r-2)}(n) \big)^{\frac{a(r)}{m}}$. 
\end{theorem}

Hence, the natural hypergraph analogue of the Erd\H{o}s-Hajnal conjecture holds for forbidden order-size pairs. 
Note that for $f = \binom{m}{r}$, the problem of estimating the minimum possible independence number of an $n$-vertex $r$-graph with no $(m,\binom{m}{r})$-subset is equivalent to estimating the off-diagonal Ramsey number $R_r(m,t)$ (as a function of $t$). From known lower bounds on these Ramsey numbers (see \cite{MS_steppingup,MS_small_r}), we get that already for $m = r+2$ (and hence for all $m \geq r+2$), there exist $n$-vertex $r$-graphs with no 
$(m,\binom{m}{r})$-subset and no homogeneous set of size $\big(\log_{(r-2)}(n) \big)^{O(1)}$. This means that the bound in Theorem \ref{main-result-1} is tight. We believe that the assumption $m \geq 5r^2$ in Theorem \ref{main-result-1} is an artifact of our proof method, and the theorem should hold for all $m$. We can prove this for $r=3$, see Proposition \ref{3-graphs}. We can also resolve the first non-trivial case $m=r+1$:
\begin{proposition}\label{thm:m=r+1}
    There is $a > 0$ such that for every $r \geq 3$ and $0 \leq f \leq r+1$, and for all large enough $n$, every $n$-vertex $r$-graph with no $(r+1,f)$-subset contains a homogeneous set of size at least 
    $\big(\log_{(r-2)}(n) \big)^{a}$. 
\end{proposition}

Our second result is concerned with the following question: Given an $n$-vertex $r$-graph $G$ with no homogeneous sets of size polynomial in $n$, 
what can we say about order-size pairs $(m,f)$ in $G$? The following well-known example shows that one cannot guarantee any specific pair $(m,f)$ in such a hypergraph. Take a random $n$-vertex tournament $T$ and consider the $3$-graph $G$ whose edges are the cyclic triangles of $T$. Then 
with high probability $G$ only has homogeneous sets of size $O(\log n)$. Moreover, any $m$-vertex tournament has at most $\binom{m}{3} - m\binom{\frac{m-1}{2}}{2} = \frac{m(m^2-1)}{24}$ cyclic triangles.
Therefore every set of $m$ vertices of $G$ spans at most $\frac{m(m^2-1)}{24}$ edges. 
As $\frac{m(m^2-1)}{24} < \frac{1}{2}\binom{m}{3}$ for $m \geq 6$, we get that for every $0 \leq f \leq \binom{m}{3}$, either $G$ or its complement $\bar{G}$ has no $(m,f)$-subset. It is worth noting that for $m=4,f=2$, the second author and Tomon \cite{GT} showed that a $3$-graph with no homogeneous sets of size $n^{\Omega(1)}$ does in fact contain a $(4,2)$-subset. 

The above discussion shows that one cannot guarantee any specific order-size pair in a hypergraph with no polynomial-size homogeneous sets. Instead, it is natural to ask how many different sizes on $m$ vertices are guaranteed to exist. For an $r$-graph $G$ and $m \geq r$, let $s(G;m)$ be the number of $0 \leq f \leq \binom{m}{r}$ such that $G$ contains an $(m,f)$-subset. 
We are interested in the minimum possible value of $s(G;m)$ over all $n$-vertex $r$-graphs with no homogeneous sets of size $n^{\Omega(1)}$. 
In \cite{ABGMW}, a conjecture on this value was made. To state this conjecture, we need to recall a well-studied function coming from another problem of Erd\H{o}s and Hajnal \cite{EH_offdiagonal}. Define $g_r(m)$ inductively as follows: $g_r(m) = 0$ for $m \leq r-1$, and otherwise 
$g_r(m) = \max_{m_1 + \dots + m_r = m} \left( \prod_{i=1}^r m_i + g_r(m_1) + \dots + g_r(m_r) \right)$. Namely, $g_r(m)$ is the maximum number of edges in an $m$-vertex $r$-graph obtained by partitioning the vertices into $r$ parts, taking a complete $r$-partite $r$-graph between the parts, and applying the same construction recursively inside each of the parts. It is easy to see that $g_r(m) = \Theta_r(m^r)$.

Erd\H{o}s and Hajnal \cite{EH_offdiagonal} proved that if an $n$-vertex $r$-graph has no $(m,f)$-subset for any $f \geq g_r(m)$, then it has a homogeneous set of size $n^{\Omega(1)}$. Conversely, Mubayi and Razborov \cite{MR} (settling a conjecture of Erd\H{o}s and Hajnal \cite{EH_offdiagonal}) proved that for $r \geq 4$ this is best possible, by constructing $n$-vertex $r$-graphs with no $(m,f)$-subset for any $f \geq g_r(m)+1$, and only having homogeneous sets of size $O(\log n)$ (the case $r=3$ for certain values of $m$ was handled previously in \cite{CFS_hyperRamsey}). These results motivated the conjecture, suggested in \cite{ABGMW}, that if $G$ is an $n$-vertex $r$-graph with no homogeneous sets of size $n^{\Omega(1)}$, then $s(G;m) \geq g_r(m)+1$. The construction of Mubayi and Razborov shows that this would be best possible (at least for $r \geq 4$, though it is believed that an analogous result should also hold for $r=3$). 
As it turns out, however, this conjecture is false: In the appendix we describe a construction by J. Fox (personal communication) showing that for a specific value of $m$, namely $m = 2r$, there exist $r$-graphs with only homogeneous sets of size $O(\log n)$ and with $s(G;2r) \leq g_r(2r)$. It seems plausible, however, that the conjecture holds asymptotically:
\begin{conjecture}\label{conj:many sizes}
    Every $n$-vertex $r$-graph with no homogeneous sets of size $n^{\Omega(1)}$ satisfies $s(G;m) \geq (1-o_m(1))g_r(m)$.
\end{conjecture}
\noindent
Here we make progress towards Conjecture \ref{conj:many sizes} for $3$-uniform hypergraphs, by showing that $s(G;m) = \Theta(m^3)$. 

\begin{theorem} \label{main-result-2}
    There is an absolute constant $c > 0$ such that the following holds. For every $m \geq 3$ there is $\varepsilon = \varepsilon(m) > 0$ such that every large enough $n$-vertex $3$-graph $G$ with no homogeneous set of size $n^\varepsilon$ satisfies $s(G;m) \geq cm^3$.
\end{theorem}
\noindent
As a further step towards Conjecture \ref{conj:many sizes}, it would be interesting to extend Theorem \ref{main-result-2} to all $r \geq 4$.

By using a recent general result of Buci\'c, Fox and Pham \cite{BFP}, we can obtain a ``polynomial-R\"odl-version" of Theorem \ref{main-result-2}. Namely, by combining Theorem \ref{main-result-2} with \cite[Theorem 19]{BFP}, we get that for every $m \geq 3$ there is $K = K(m)$ such that for every $\gamma > 0$ and for every $n$-vertex $3$-graph $G$, if $s(G;m) < cm^3$ (where $c > 0$ is the same constant as in Theorem \ref{main-result-2}) then $G$ contains a vertex-set $S$, $|S| \geq \gamma^K n$, such that $S$ has density at most $\gamma$ or at least $1-\gamma$. By taking $\gamma \approx n^{-\frac{2}{2K+1}}$, one recovers Theorem \ref{main-result-2} with $\varepsilon = \frac{1}{2K+1}$ (this is the standard reduction from a polynomial-R\"odl result to an Erd\H{o}s-Hajnal result).

Finally, we note that the question of the existence of large homogeneous sets in graphs with few order-size pairs was also studied in \cite{AB}, where it was shown that if a (2-uniform) graph $G$ has only at most $\ell$ different sizes of induced $m$-vertex subgraphs, where $2\ell \leq m \leq n-2\ell$, then $G$ has a homogeneous set of size $n-\ell+1$.

\section{Proof of Theorem \ref{main-result-1} and Proposition \ref{thm:m=r+1}}
Throughout this section, it will be convenient to identify $r$-graphs with $\{0,1\}$-colored complete $r$-graphs (where an edge being present corresponds to color $1$). Section \ref{subsec:prelim} contains some preliminaries. We give an overview of the proofs in Section \ref{subsec:overview}, and then prove Theorem \ref{main-result-1} in Sections \ref{subsec:3-graphs}-\ref{subsec:general case} and Proposition \ref{thm:m=r+1} in Section \ref{subsec:r+1}. 

\subsection{Preliminaries}\label{subsec:prelim}

The following theorem is the well-known stepping-down argument of Erd\H{o}s and Rado \cite{ER}. 
An ordered set is simply a set with a linear order, which we denote by $\leq$.




\begin{theorem}[\cite{ER}]\label{thm:ER_basic}
Let $r \geq 3$, let $n,\ell \in \mathbb{N}$ with $2^{\binom{\ell-1}{r-1}} \leq n$, let 
$V$ be an ordered set of size $n$, and let
$c : \binom{V}{r} \rightarrow \{0,1\}$. Then there is a subset $X \subseteq V$ of size $\ell$, and there is $\chi : \binom{X}{r-1} \rightarrow \nolinebreak \{0,1\}$, such that 
$c(x_1,\dots,x_r) = \chi(x_1,\dots,x_{r-1})$ for all
$x_1,\dots,x_r \in X$ with
$x_1 < \dots < x_r$.
\end{theorem}

By applying Theorem \ref{thm:ER_basic} $r-2$ times, one obtains a subset $X \subseteq V$ and $\chi : \binom{X}{2} \rightarrow \{0,1\}$ such that $c(x_1,\dots,x_r) = \chi(x_1,x_2)$ for all 
$x_1 < \dots < x_r$ in $X$. We need the following generalization, which is obtained by repeatedly applying Theorem \ref{thm:ER_basic} and (possibly) reversing the order of the vertices between applications.  
Here the logarithms are base $2$.


\begin{theorem}\label{erdos-rado}
    Let $r \geq 3$ and $1 \leq k \leq r-1$, let
    $n,\ell \in \mathbb{N}$ with 
    $\ell \leq \sqrt{\frac{2}{3}\log_{(r-2)}(n)}$, let $V$ be an ordered set of size $n$, and let 
    $c : \binom{V}{r} \rightarrow \{0,1\}$.
    Then there is a subset $X \subseteq V$ of size $\ell$, and there is $\chi: \binom{X}{2} \rightarrow \{0,1\}$, such that
    $c(x_1,\dots,x_r) = \chi(x_k,x_{k+1})$ for all $x_1,\dots,x_r \in X$ with
$x_1 < \dots < x_r$.
\end{theorem}

\begin{proof}

     Observe that the statement of the theorem for $r-k$ is equivalent to the statement for $k$, by reversing the order on $V$. Hence, we may assume that $k \leq \lfloor \frac{r}{2} \rfloor$. 

     We proceed by induction on $r$. For the induction to work, we will actually prove that the conclusion of the theorem holds if 
     \begin{equation}\label{eq:ell ER}
      \ell \leq 
     \begin{cases}
         \sqrt{2\log n} & r = 3, \\
         \sqrt{\frac{2}{3}\log_{(r-3)}\left((r-1)!\log n\right)} & r \geq 4.
     \end{cases}
     \end{equation}
     It is easy to see that both expressions in \eqref{eq:ell ER} are larger than $\sqrt{\frac{2}{3}\log_{(r-2)}(n)}$, so this implies the theorem.
     
     For the base case $r=3$, the only possible value of $1 \leq k \leq \lfloor \frac{r}{2} \rfloor$ is $k=1$, and this case holds by Theorem \ref{thm:ER_basic}, because if $\ell \leq \sqrt{2\log n}$ then $2^{\binom{\ell-1}{2}} \leq n$. 
     
     Let now $r \geq 4$ and let $1 \leq k \leq \lfloor \frac{r}{2} \rfloor$. Let $t$ be the largest integer satisfying $2^{\binom{t-1}{r-1}} \leq n$. 
     By the maximality of $t$, we have 
     $\frac{t^{r-1}}{(r-1)!} \geq \binom{t}{r-1} > \log n$, and hence 
     $t \geq \left( (r-1)! \log n \right)^{\frac{1}{r-1}}$. By Theorem \ref{thm:ER_basic}, there is $X' \subseteq V$ of size $t$, and there is $c' : \binom{X'}{r-1} \rightarrow \{0,1\}$, such that 
     $c(x_1,\dots,x_r) = c'(x_1,\dots,x_{r-1})$ for all $x_1 < \dots < x_r$ belonging to $X'$. Now we want to apply the induction hypothesis for $r-1$ to the coloring $c'$. 
     Note that $k \leq \lfloor \frac{r}{2} \rfloor \leq r-2$, so we may apply the induction hypothesis with this value of $k$.
     We need to verify that $\ell$ satisfies \eqref{eq:ell ER} with $r-1$ in place of $r$ and $t$ in place of $n$. Suppose first that $r=4$. Then we know that $\ell \leq \sqrt{\frac{2}{3}\log(6\log n)}$, and we need to verify that $\ell \leq \sqrt{2\log t}$. As $t \geq (6\log n)^{1/3}$, we have 
     $\sqrt{2\log t} \geq \sqrt{\frac{2}{3}\log(6\log n)} \geq \ell$, as required. Similarly, for $r \geq 5$ we have that 
     $$
     \log_{(r-4)}\left((r-2)!\log t \right) \geq 
     \log_{(r-4)}\left((r-2)! \cdot \log \left( \left( (r-1)! \log n \right)^{\frac{1}{r-1}} \right) \right) \geq \log_{(r-3)}((r-1)!\log n),
     $$
     where the last inequality uses that $\frac{(r-2)!}{r-1} \geq 1$. Since we assume that  \eqref{eq:ell ER} holds, we get that 
     $$
     \ell \leq \sqrt{\frac{2}{3}\log_{(r-4)}\left((r-2)!\log t\right)} \; ,
     $$
     as required. Now, by the induction hypothesis, there is a subset $X \subseteq X'$ of size $\ell$, and there is
     $\chi : \binom{X}{2} \rightarrow \{0,1\}$, such that 
     $c'(x_1,\dots,x_{r-1}) = \chi(x_k,x_{k+1})$ for every $x_1 < \dots < x_{r-1}$ in $X$. Hence, for every $x_1 < \dots < x_r$ in $X$, it holds that 
     $c(x_1,\dots,x_r) = c'(x_1,\dots,x_{r-1}) = \chi(x_k,x_{k+1})$, as required. 
\end{proof}


Using Theorem \ref{erdos-rado} (with $k=1$), we can reduce Theorem \ref{main-result-1} to a problem about ordered (2-uniform) graphs, namely, we consider the graph given by the coloring $\chi$. We will need some facts about the Erd\H{o}s-Hajnal conjecture for ordered graphs. An ordered graph $H$ is said to satisfy the {\em ordered EH-property with constant $c > 0$} if every $n$-vertex ordered graph with no induced copy of $H$ contains a homogeneous set of size at least $n^c$. (We refer the reader to \cite{nguyen2023} for new results on the ordered Erd\H{o}s-Hajnal conjecture.)

For ordered graphs $H,F_1,\dots,F_h$, where $H$ has vertices $1,\dots,h$, the {\em substitution} $H[F_1,\dots,F_h]$ is the ordered graph obtained by substituting $F_i$ into vertex $i$ for each 
$i \in [h] = V(H)$; namely, for all $1 \leq i < j \leq h$, the bipartite graph between $F_i,F_j$ is complete if $ij \in E(H)$ and empty if $ij \notin E(H)$. For unordered graphs, Alon, Pach and Solymosi \cite{alonEHproperty} showed that if $H,F_1,\dots,F_h$ have the EH-property then so does $H[F_1,\dots,F_h]$. This is also true for ordered graphs with essentially the same proof.
\begin{lemma}[cf.~Lemma 3.4 in \cite{nguyen2023}]\label{lem:substitution}
    If ordered graphs $H,F_1,\dots,F_h$ have the ordered EH-property, then so does $H[F_1,\dots,F_h]$.
\end{lemma}

We now describe some simple ordered graphs which have the ordered EH-property. This will be used in the proof of Theorem \ref{main-result-1}.
First, let $F_0$ denote the ordered graph with vertices $1,2,3$ and a single edge $13$; so $F_0$ is the complement of the monotone path with edges $12,23$. We will use the known fact that $F_0$ has the ordered EH-property. 

\begin{lemma}\label{lem:comparability}
The ordered graph $F_0 = (\{1,2,3\}, \{13\})$ has the ordered EH-property.
\end{lemma}
\begin{proof}[Proof sketch]
    If an ordered graph $G$ has no induced copy of the monotone path $\bar{F_0} = \{12,23\}$, then $G$ is a comparability graph. It is well-known that comparability graphs are perfect (see \cite[Chapter 3]{Golumbic}) and thus contain a homogeneous set of size at least $\sqrt{n}$. It follows that $\bar{F_0}$ (and hence $F_0$) has the ordered EH-property.
\end{proof}

A {\em monotone star} is an ordered star whose center comes before its leaves in the vertex-order; see Figure \nolinebreak \ref{figure-star-forest-1}.
A {\em monotone star forest} is the
disjoint union of monotone stars $S_1,\dots,S_m$, such that $S_1 < \dots < S_m$ in the vertex-order; see Figure \nolinebreak \ref{figure-star-forest-1}. Note that a star may have no leaves, in which case its center is isolated. 

\begin{figure}[h]
\centering
\begin{tikzpicture}[thick,
  bnode/.style={draw,circle,fill,inner sep=1pt},every fit/.style={ellipse,draw,inner sep=-2pt,minimum height=0.4cm}
]

\node[bnode,xshift=0.75*1cm] (1) {};
\node[bnode,xshift=0.75*2cm] (2) {};
\node[bnode,xshift=0.75*3cm] (3) {};
\path (2) -- node[auto=false]{\ldots} (3);
\node[bnode,xshift=0.75*4cm] (4) {};
\node[bnode,xshift=0.75*5cm] (5) {};
\node[bnode,xshift=0.75*6cm] (6) {};
\path (5) -- node[auto=false]{\ldots} (6);
\node[bnode,xshift=0.75*8cm] (7) {};
\path (6) -- node[auto=false]{\ldots} (7);
\node[bnode,xshift=0.75*9cm] (8) {};
\node[bnode,xshift=0.75*10cm] (9) {};
\path (8) -- node[auto=false]{\ldots} (9);

\draw (1) to [out=90,in=90,looseness=1] (2);
\draw (1) to [out=90,in=90,looseness=1] (3);
\draw (4) to [out=90,in=90,looseness=1] (5);
\draw (4) to [out=90,in=90,looseness=1] (6);
\draw (7) to [out=90,in=90,looseness=1] (8);
\draw (7) to [out=90,in=90,looseness=1] (9);

\end{tikzpicture}
\caption{A monotone star forest}\label{figure-star-forest-1}
\end{figure}
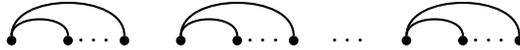

A nested star forest is an ordered star forest consisting of stars $S_1,\dots,S_t$, where $S_i$ has center $c_i$ and leaf-set $L_i$, such that $L_1 < L_2 < \dots < L_t < c_t < \dots < c_1$ in the vertex-order; see Figure \ref{figure-star-forest-2}. We note that the sets $L_i$ are allowed to be empty (in which case $c_i$ is isolated). 

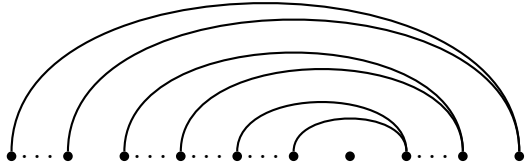
\begin{figure}[h]
\centering
\begin{tikzpicture}[thick,
  bnode/.style={draw,circle,fill,inner sep=1pt},every fit/.style={ellipse,draw,inner sep=-2pt,minimum height=0.4cm}
]

\node[bnode,xshift=0.75*6cm] (6) {};
\node[bnode,xshift=0.75*7cm] (7) {};
\path (6) -- node[auto=false]{\ldots} (7);
\node[bnode,xshift=0.75*8cm] (8) {};
\node[bnode,xshift=0.75*9cm] (9) {};
\path (8) -- node[auto=false]{\ldots} (9);
\node[bnode,xshift=0.75*10cm] (10) {};
\path (9) -- node[auto=false]{\ldots} (10);
\node[bnode,xshift=0.75*11cm] (11) {};
\path (10) -- node[auto=false]{\ldots} (11);
\node[bnode,xshift=0.75*12cm] (12) {};
\node[bnode,xshift=0.75*13cm] (13) {};
\node[bnode,xshift=0.75*14cm] (14) {};
\path (13) -- node[auto=false]{\ldots} (14);
\node[bnode,xshift=0.75*15cm] (15) {};

\draw (6) to [out=90,in=90,looseness=1] (15);
\draw (7) to [out=90,in=90,looseness=1] (15);
\draw (8) to [out=90,in=90,looseness=1] (14);
\draw (9) to [out=90,in=90,looseness=1] (14);
\draw (10) to [out=90,in=90,looseness=1] (13);
\draw (11) to [out=90,in=90,looseness=1] (13);

\end{tikzpicture}
\caption{A nested star forest}\label{figure-star-forest-2}
\end{figure}

\begin{lemma}\label{lem:star forests EH}
Every monotone star forest and every nested star forest have the ordered EH-property.  
\end{lemma}
\begin{proof}
Using Lemmas \ref{lem:substitution} and \ref{lem:comparability}, it suffices to show that every monotone star forest and every nested star forest can be built from empty graphs, cliques, and $F_0$ using substitution. 
A monotone star with $\ell$ leaves can be obtained from a single edge by substituting an empty graph of size $\ell$ into one of the vertices. Furthermore, a monotone star forest can be obtained by substituting monotone stars into the vertices of an empty graph. This proves the claim for monotone star forests. 

A nested star forest 
$L_1 < L_2 < \dots < L_t < c_t < \dots < c_1$ can be obtained from the graph $F_0$ by substituting an empty graph of size $|L_1|$ into the vertex $1$, the star forest 
$L_2 < \dots < L_t < c_t < \dots < c_2$ (which has one component less) into the vertex $2$, and having vertex $3$ play the role of $c_1$. Thus, every nested star forest can be obtained from $F_0$ by repeated substitutions.
\end{proof}

\subsection{Proof overview}\label{subsec:overview}

We give an overview of the proofs of Theorem \ref{main-result-1} and Proposition \ref{thm:m=r+1}.  
For Theorem \ref{main-result-1}, given an $r$-graph $G$, we consider the ordered graph given by Theorem \ref{erdos-rado} (with $k=1$), i.e., the graph with vertex-set $X$ and edge-set 
$\{xy \in \binom{X}{2} : \chi(x,y) = 1\}$. For convenience, write $X = \{v_1 < \dots < v_{\ell}\}$.
Assuming that this graph does not have large (i.e., of size polynomial in $\ell$) homogeneous sets, we would like to find a configuration which corresponds to an $(m,f)$-subset in the original $r$-graph $G$. This configuration consists of $m-r+2$ vertices $u_1,\dots,u_{m-r+2} \in \{v_1,\dots,v_{\ell-r+2}\}$, and it should satisfy that by adding $r-2$ vertices $u_{m-r+3},\dots,u_m > u_{m-r+2}$, we get $m$ vertices $u_1,\dots,u_m$ that span exactly $f$ edges in $G$. Recall that whether or not an edge $u_{i_1},\dots,u_{i_r}$ is present in $G$ depends only on $\chi(u_{i_1},u_{i_2})$. Thus, a pair $u_i,u_j$ ($i < j$) comes with a ``weight" $w_U(u_i,u_j)$ relative to the set $U = \{u_1,\dots,u_m\}$; this weight is the total number of $r$-tuples in $\{u_1,\dots,u_m\}$ in which $u_i,u_j$ are the first two vertices, so $w_U(u_i,u_j) = \binom{m-j}{r-2}$. So we see that if $u_iu_j$ is an edge in the ordered graph, then this edge contributes exactly $w_U(u_i,u_j)$ hyperedges to the $r$-graph induced by $U = \{u_1,\dots,u_m\}$. Therefore, our goal is to find in the ordered graph on 
$\{v_1,\dots,v_{\ell-r+2}\}$ an $(m-r+2)$-vertex graph whose total weight (i.e., the sum of weights of its edges) is exactly $f$. 

To take advantage of the above, we need to find an $(m-r+2)$-vertex ordered graph $H$ which has weight $f$ and satisfies the ordered EH-property.
We will construct such graphs $H$ by induction on $m$. As we shall see, if $m$ is large enough (say $m \geq 2r$) then the induction step can be done by adding an isolated vertex or a universal vertex (i.e., a vertex adjacent to all other vertices) to obtain a graph on $m$ vertices from a graph on $m-1$ vertices. (This is similar to the approach used in \cite{ABGMW} to prove Theorem \ref{main-result-1} for uniformity $r=2$.) 
However, for small values of $m$, this no longer works. 
In fact, the entire approach cannot work for (say) $m = r+1$ and $r \geq 5$, because the weights $w_U(u_i,u_j)$ of a set $U = \{u_1,u_2,u_3\}$ of size $m-r+2 = 3$ are $w_U(u_1,u_2) = r-1$ and $w_U(u_1,u_3) = w_U(u_2,u_3) = 1$, so there is no $3$-vertex ordered graph $H$ with weight $3 \leq f \leq r-2$. (Note that Proposition \ref{thm:m=r+1} does in fact establish the case $m=r+1$, using a modification of this argument; see below.) 

Thus, the main challenge in the proof of Theorem \ref{main-result-1} is to prove an induction basis. 
For $r=3$, we will be able to handle all cases where the induction step is not applicable, and thus get the result for all values of $m \geq 3$.
However, for general $r$ the situation is considerably more complicated. 
First, as mentioned above, if $m$ is too small then there are values of $f$ which cannot be realized by any $(m-r+2)$-vertex ordered graph. 
If $m$ is somewhat large, say $m \geq 2r$, then one can show that every choice of $0 \leq f \leq \binom{m}{r}$ is in fact realizable, but we cannot show that these graphs have the ordered EH-property. 
Finally, for larger $m$, namely $m \geq 5 r^2$, we can indeed construct a graph satisfying the ordered EH-property and having weight $f$ for any given $0 \leq f \leq \binom{m}{r}$, establishing an induction \nolinebreak basis. 

The proof of Proposition \ref{thm:m=r+1} applies Theorem \ref{erdos-rado} with different values of $k$. This changes the weights: for a set $U = \{u_1 < \dots < u_m\}$ and a pair of indices $i,j$ with $k \leq i < j \leq m-r+k+1$, the weight of $u_iu_j$ is now $w_U(u_iu_j) = \binom{i-1}{k-1}\binom{m-j}{r-k-1}$. For $m=r+1$, choosing $k=f$ allows one to obtain an ordered graph on $m-r+2 = 3$ vertices with weight $f$. 
It is worth noting that already for $m=r+2$, there are choices of $r,f$ for which there does not exist any ordered graph on $m-r+2=4$ vertices with weight $f$, no matter the $k$ with which we apply Theorem \ref{erdos-rado}. One such example is $r=10$ (so $m=12$) and $f=33$. (The non-existence of an ordered graph with weight $f$ can be verified by hand for each choice of $k$; by symmetry under $k \mapsto r-k$, it suffices to consider $1 \leq k \leq 5$.)

\subsection{Proof of Theorem \ref{main-result-1}: $3$-graphs} \label{subsec:3-graphs}

We prove Theorem \ref{main-result-1} in the case $r=3$ with the following quantitative bound (where the logarithm is base 2).

\begin{proposition} \label{3-graphs}
    For every $m \geq 3$, $0 \leq f \leq \binom{m}{3}$ and large enough $n$, every $n$-vertex 3-graph with no $(m,f)$-subset contains a homogeneous set of size at least 
    $h := \big\lfloor \frac{1}{2}\log^{\frac{1}{2(m-2)}}(n) \big\rfloor$.
\end{proposition}

\begin{proof}
    We prove the contra-positive: Let $G$ be an $n$-vertex $3$-graph with no homogeneous set of size $h$ (as defined in the statement).
    Our goal is to show that $G$ has an $(m,f)$-subset. 
    We identify $G$ with the corresponding coloring $c:E(K_n^{(3)}) \to \{0,1\}$. 
    By Theorem \ref{erdos-rado} with $k = 1$ and  
    $\ell := \lceil \frac{1}{2}\log^{\frac{1}{2}}(n) \rceil + 1$, there exist $v_1,\dots,v_{\ell} \in V(G)$ and colors $\chi(v_i,v_j) \in \{0,1\}$ for $1 \leq i < j \leq \ell-1$ such that $c(v_iv_jv_k) = \chi(v_i,v_j)$ for all $1\leq i < j < k \leq \ell$. 
    Let $G'$ be the ordered graph on $V(G') = \{v_1,\dots,v_{\ell-1}\}$ with edge-set $\{v_iv_j : \chi(v_i,v_j) = 1\}$. 
    Note that if $S \subseteq \{v_1,\dots,v_{\ell-1}\}$ is a homogeneous set in $G'$ then 
    $S \cup \{v_{{\ell}}\}$ is a homogeneous set in $G$. Hence, $G'$ does not contain a homogeneous set of size $h$.


    For an $(m-1)$-element subset $U = \{u_1 < \dots < u_{m-1}\} \subseteq V(G')$ and indices $1 \leq i < j \leq m-1$, define $w_U(u_i,u_j) := m-j$. We say that $U = \{u_1 < \dots < u_{m-1}\}$ is a {\em weighted $(m,f)$-subset} in $G'$ if 
    $$
    \sum_{1 \leq i < j \leq m-1} w_U(u_i,u_j) \cdot \chi(u_i,u_j) = f.
    $$
    Observe that the number of edges spanned by $\{u_1,\dots,u_{m-1},v_{\ell}\}$ in $G$ is
    $$e_G(U \cup \{v_{\ell}\}) = 
    \sum_{xyz \in \binom{U \cup \{v_{\ell}\}}{3}} c(xyz) = 
    \sum_{1 \leq i < j \leq m-1} w_U(u_i,u_j) \cdot \chi(u_i,u_j).
    $$ 
    Thus, if $U$ is a weighted $(m,f)$-subset in $G'$ then $U \cup \{v_{\ell}\}$ is an $(m,f)$-subset in $G$.
    So to complete the proof, it remains to show that $G'$ has a weighted-$(m,f)$-subset $U$. This will be done using the following lemma.

\begin{lemma} \label{3-graphs induction}
    Let $m \geq 3$, $0 \leq f \leq \binom{m}{3}$ and $h \in \mathbb{N}$, and let $G'$ be an ordered graph with $|V(G')| \geq h^{m-2}$. Then $G'$ has a weighted $(m,f)$-subset or a homogeneous set of size $h$. 
\end{lemma}
    \noindent
    As $G'$ has no homogeneous set of size $h$, and as 
    $$
    |V(G')| = \ell-1 = 
    \left\lceil \frac{1}{2}\log^{\frac{1}{2}}(n) \right\rceil \geq 
    \left\lfloor \frac{1}{2}\log^{\frac{1}{2(m-2)}}(n) \right\rfloor ^{m-2} = h^{m-2},
    $$
    Lemma \ref{3-graphs induction} implies that $G'$ has a weighted $(m,f)$-subset. This proves Proposition \ref{3-graphs} given Lemma \nolinebreak \ref{3-graphs induction}. 
\end{proof}


\begin{proof}[Proof of Lemma \ref{3-graphs induction}]
    By taking complements we can assume without loss of generality that $f \leq \frac{1}{2}\binom{m}{3}$. We proceed by induction on $m \geq 3$. For $m=3$, we have $f=0$, and either $G'$ is a clique (of size at least $h^{m-2} = h$) or it has a non-edge, which gives a weighted $(3,0)$-subset.  
    
    Let $m>3$. 
    For a vertex $v \in V(G')$, let $\bar{N}^+(v)$ denote the forward non-neighborhood of $v$, i.e., the set of vertices $u > v$ such that $uv \notin E(G')$. 
    If $|\bar{N}^+(v)| < h^{m-3}$ for every vertex $v \in V(G')$, then we can greedily build a clique of size $h$ by repeatedly adding the smallest (in the vertex-order) remaining vertex to the clique and deleting all of its forward non-neighbors. At each step we delete at most $h^{m-3}-1$ vertices and add one vertex to the clique (for a total of $h^{m-3}$ deleted vertices), so at the end we get a clique of size at least $\frac{|V(G')|}{h^{m-3}} \geq h$, as required. 
    
    Suppose now that there is $v \in V(G')$ with $|\bar{N}^+(v)| \geq h^{m-3}$. Setting $N := \bar{N}^+(v)$,
    we apply the induction hypothesis for $m-1$ to $G'[N]$.
    As $f \leq \frac{1}{2} \binom{m}{3}$, we have $f \leq \binom{m-1}{3}$ 
    unless $m=4,f=2$ or $m=5,f=5$. We will consider these two cases separately below. For now, we assume that $f \leq \binom{m-1}{3}$.
    By the induction hypothesis for $m-1$, $G'[N]$ has a homogeneous set of size $h$ or a weighted 
    $(m-1,f)$-subset $U$. 
    Suppose we are in the latter case. There are no edges from $v$ to $U$, as $U \subseteq N$. Also, $w_{\{v\} \cup U}(u,w) = w_U(u,w)$ for all $u,w \in U$, because $v$ comes before $U$ in the vertex order.  
    It follows that $\{v\} \cup U$ is a weighted $(m,f)$-subset in $G'$, as required. We now handle the two remaining cases:

    \begin{itemize}
        \item 
        {\bf Case $m=4,f=2$.} For a vertex $v$, let $\bar{N}^-(v)$ denote the backward non-neighborhood of $v$, i.e., the set of vertices $u < v$ such that $uv \notin E(G')$.
        By the same argument as above, either $G'$ has a clique of size $h$ or there exists $v \in V(G')$ with $|\bar{N}^-(v)| \geq h^{m-3} = h$. If $\bar{N}^-(v)$ is an independent set then we are done, and else there are $u_1,u_2 \in \bar{N}^-(v)$ with $u_1u_2 \in E(G')$. Now $U=\{u_1,u_2,v\}$ is a weighted $(4,2)$-subset in $G'$; indeed, the only edge in $G'[U]$ is $u_1u_2$, and it has weight $w_U(u_1,u_2)=2$.
        \item {\bf Case $m=5,f=5$.} With the same argument as before, we can find $v \in V(G')$ with 
        $|\bar{N}^-(v)| \geq h^{m-3} = h^2$.
        By the same argument inside $N := \bar{N}^-(v)$ (for the complement of $G'$), either we can find an independent set of size $h$ (in which case we are done) or there is a vertex $v' \in N$ such that $v'$ has at least $h$ forward neighbors inside $N$. Namely, the set $N' = \{u \in N : u > v', uv' \in E(G')\}$ satisfies $|N'| \geq h$. 
        If $N'$ is a clique then we are done, and else there is a nonedge $u_1u_2$ with $u_1,u_2 \in N'$. 
        Now $U=\{v',u_1,u_2,v\}$ is a weighted $(5,5)$-subset in $G'$; indeed, the only two edges in $G'[U]$ are $v'u_1,v'u_2$, and they have weight $w_U(v',u_1) = 3$ and $w_U(v',u_2)=2$.
    \end{itemize}
\end{proof}

\subsection{Proof of Theorem \ref{main-result-1}: general case} \label{subsec:general case}

Here we prove Theorem \ref{main-result-1} for $r \geq 4$. 
The proof is analogous to that of Proposition \ref{3-graphs}, with the key difference being the induction basis (which is the key challenge in proving Theorem \ref{main-result-1}). 

\begin{proof}[Proof of Theorem \ref{main-result-1}]
    The case $r=3$ follows from Proposition \ref{3-graphs}, so we can assume $r \geq 4$. 
    We will prove the theorem for a small enough $a(r) > 0$ to be chosen implicitly later.
    Let $n$ be large enough, and let $G$ be an $n$-vertex $r$-graph.
    We will show that $G$ has an $(m,f)$-subset or a homogeneous set of size at least 
    $$
    \frac{1}{2}\log_{(r-2)}^{\frac{a(r)}{m}}(n).
    $$
    (The constant $\frac{1}{2}$ can be absorbed by decreasing $a(r)$.)
    
    We identify $G$ with the corresponding coloring $c:E(K_n^{(r)}) \to \{0,1\}$. 
    Set 
    $$
    \ell := \Big\lceil \frac{1}{2} \log^{\frac{1}{2}}_{(r-2)}(n) \Big\rceil + r-2,
    $$
    and apply
    Theorem \ref{erdos-rado} with $k=1$ to get vertices $v_1,\dots,v_{\ell} \in V(G)$ and colors $\chi(v_i,v_j) \in \{0,1\}$ for $1 \leq i < j \leq \ell-r+2$ such that $c(v_{i_1},\dots,v_{i_r}) = \chi(v_{i_1},v_{i_2})$ for all $1 \leq i_1 < i_2 < \dots < i_r \leq \ell$. Let $G'$ be the ordered graph on $V(G') := \{v_1,\dots,v_{\ell-r+2}\}$ with edge-set 
    $\{v_iv_j : \chi(v_iv_j) = 1\}$. 

    Consider an $(m-r+2)$-element subset $U = \{u_1 < \dots < u_{m-r+2}\} \subseteq V(G')$. For a pair of indices $1 \leq i < j \leq m-r+2$, define $w_U(u_i,u_j) := \binom{m-j}{r-2}$. 
    Observe that 
    \begin{equation}\label{eq:weighted (m,f)-subset}
    e_G(U \cup \{v_{\ell-r+3},\dots,v_{\ell}\}) = 
    \sum_{1 \leq i < j \leq m-r+2} w_U(u_i,u_j) \cdot \chi(u_i,u_j).
    \end{equation}
    We say that 
    $U$ is a {\em weighted $(m,f)$-subset} in $G'$ if the right-hand side of \eqref{eq:weighted (m,f)-subset} equals $f$.
    Thus, if $U$ is a weighted $(m,f)$-subset in $G'$ then $U \cup \{v_{\ell-r+3},\dots,v_{\ell}\}$ is an $(m,f)$-subset in $G$. 
    The main step in the proof is the following lemma, which is analogous to Lemma \ref{3-graphs induction}.

    \begin{lemma} \label{r-graphs induction}
    For every $r \geq 4$, there is a constant $a_0(r)>0$ such that for every $m \geq m_0(r) := 5 r^2$, $0 \leq f \leq \binom{m}{r}$ and $h \in \mathbb{N}$, every ordered graph $G'$ with $|V(G')| \geq h^{m - m_0(r) + \frac{1}{a_0(r)}}$ contains a homogeneous set of size $h$ or a weighted $(m,f)$-subset.
    \end{lemma}
    
    Let us complete the proof of the theorem using Lemma \ref{r-graphs induction}.
    If $G'$ has a weighted $(m,f)$-subset then we are done. Otherwise, $G'$ has a homogeneous set $S$ of size 
    $$
    |S| \geq h := \lfloor |V(G')|^{\alpha} \rfloor = \lfloor (\ell-r+2)^\alpha \rfloor,
    $$
    where 
    $\alpha := \big(m - m_0(r) + \frac{1}{a_0(r)}\big)^{-1}$. Observe that $S \cup \{v_{\ell-r+3},\dots,v_{\ell}\}$ is a homogeneous set in $G$. Hence, $G$ has a homogeneous set of size
    $$
    \lfloor (\ell-r+2)^\alpha \rfloor + r-2 \geq
    (\ell-r+2)^\alpha \geq 
    \frac{1}{2} \log_{(r-2)}^{\frac{a(r)}{m}}(n),
    $$
    using our choice of $\ell$ and $\alpha$ and assuming that $a(r)$ is small enough as a function of $r$.
    This completes the proof. 
\end{proof}

It remains to prove Lemma \ref{r-graphs induction}. The proof is by induction on $m$, with the base case given by the following lemma. 
\begin{lemma} \label{r-graphs induction base case}
    For every $r \geq 4$,  there is a constant $a_0(r)>0$ such that for $m := 5 r^2$ and for every $0 \leq f \leq \binom{m}{r}$, there is an ordered graph $H$ on $U=\{u_1,\dots,u_{m-r+2}\}$ with total weight $w_U(H) := \sum_{e \in E(H)} w_U(e) = f$ such that $H$ satisfies the ordered EH-property with constant $a_0(r)$.
\end{lemma}
\noindent
We now use Lemma \ref{r-graphs induction base case} to derive Lemma \ref{r-graphs induction}, after which we prove Lemma \ref{r-graphs induction base case}.

\begin{proof} [Proof of Lemma \ref{r-graphs induction}]
    We prove the lemma with the same constant $a_0(r)$ given by Lemma \ref{r-graphs induction base case}.
    The proof is by induction on $m$. In the base case $m=m_0(r) = 5r^2$, we apply Lemma \ref{r-graphs induction base case}. 
    Let $H$ be the ordered graph given by Lemma \ref{r-graphs induction base case}. If $G'$ has an induced copy of $H$ then this gives a weighted $(m_0(r),f)$-subset, and otherwise $G'$ 
    contains a homogeneous set of size at least $|V(G')|^{a_0(r)} \geq h$, because $H$ has the ordered EH-property with constant $a_0(r)$.
         
    For the induction step, let $m > m_0(r)$ and $0 \leq f \leq \binom{m}{r}$.
    By taking complements, we can assume without loss of generality that $f \leq \frac{1}{2}\binom{m}{r}$.
    We may also assume 
    that 
    $\frac{1}{a_0(r)}$ is an integer, so that $k := h^{(m-1) - m_0(r) + \frac{1}{a_0(r)}}$ is an integer. 
    As in the proof of Lemma \ref{3-graphs induction}, let $\bar{N}^+(v)$ denote the forward non-neighborhood of a vertex $v$, i.e., the set of vertices $u$ satisfying $u > v$ and $uv \notin E(G')$. 
    If for every vertex $v \in V(G')$ it holds that 
    $|\bar{N}^+(v)| \leq k-1$, then we can greedily build a clique of size $h$ by repeatedly adding the smallest remaining vertex (in the vertex-order) to the clique and deleting all of its forward non-neighbors. At each step we delete at most $k-1$ vertices and add one vertex to the clique (for a total of $k$ deleted vertices), so 
    the final clique has size at least $|V(G')|/k \geq h$, as required.
    
    Suppose now that there is $v \in V(G')$ with $|\bar{N}^+(v)| \geq k = h^{(m-1) - m_0(r) + \frac{1}{a_0(r)}}$, and put $N := \bar{N}^+(v)$.
    Note that $\frac{1}{2} \binom{m}{r} \leq \binom{m-1}{r}$ holds for every $m \geq 2r$, because $\frac{\binom{m}{r}}{\binom{m-1}{r}} = \frac{m}{m-r}$.
    So $f \leq \binom{m-1}{r}$, meaning that we can apply induction. 
    By the induction hypothesis for $m-1$, $G'[N]$ has a homogeneous set of size $h$ or a weighted 
    $(m-1,f)$-subset $U$. In the latter case, as there are no edges from $v$ to $U$ and 
    as $w_{\{v\} \cup U}(u,w) = w_U(u,w)$ for all $u,w \in U$ (because $v$ comes before $U$ in the vertex-order), we get that $\{v\} \cup U$ is a weighted $(m,f)$-subset in $G'$, as \nolinebreak required.
\end{proof}

\begin{proof}[Proof of Lemma \ref{r-graphs induction base case}]
    By taking complements we can assume without loss of generality that 
    $f \leq \frac{1}{2}\binom{m}{r}$. 
    Our goal is to define the required ordered graph $H$ (with vertex set $u_1,\dots,u_{m-r+2}$) and then show that $H$ has the ordered EH-property. 
    For convenience, set $w_j := \binom{m-j}{r-2}$. 
    We start by defining numbers $d_i$, $1 \leq i \leq m-r+2$, by induction on $i$. Eventually, $d_i$ will be the {\em backward-degree} of $u_i$ in the graph $H$, i.e., the number of neighbours of $u_i$ in the set $\{u_1,\dots,u_{i-1}\}$. Define $d_1 = 0$. Let $i \geq 2$ and suppose that we already defined $d_h$ for $1 \leq h \leq i-1$ and that $\sum_{h=1}^{i-1} d_h w_h \leq f$ (this certainly holds for $i=2$ because $d_1=0$). Define $d_i$ to be the maximal number $0 \leq d_i \leq i-1$ such that $\sum_{h=1}^i d_h w_h \leq f$. 
    Observe that if $d_i \leq i-2$ then $\sum_{h=1}^i d_h w_h > f - w_i$; we will frequently use this fact. We now prove some properties of the numbers $d_i$.   
    Define $i^*$ to be the minimal index $i \geq 2$ with $d_{i} \leq i-2$.


    \begin{claim}\label{d_i}
    The numbers $d_1,\dots,d_{m-r+2}$ defined above satisfy the following properties:
    \begin{enumerate}
        \item[(a)] $i^* \leq \frac{m+r}{2}$.
        \item[(b)] For every $i^* < i \leq m-r+2$, it holds that $d_i < \frac{r-2}{m-r+3-i} + 1$ and $d_i \leq i-2$. In particular, for $i^* < i \leq m-2r+5$ we have $d_i \leq 1$.
        \item[(c)] $\sum_{i=1}^{m-r+2} d_i w_i = f$.
        \item[(d)] There are indices $k,\ell$ with $i^* \leq k < \ell \leq m-2r+5$ and $\ell-k \geq 2(r-2) \log(2(r-2))$ such that $d_i=0$ for every 
        $i \in \{k+1,\dots,\ell-1\}$.\footnote{Here and in the proof of Claim \ref{d_i}, $\log$ is the natural logarithm.}
    \end{enumerate}
    \end{claim}
    \begin{proof}
        We begin with Item (a). For convenience, put $s := \lceil \frac{m+r+1}{2} \rceil$.
        First we calculate
        \begin{align*}
            \sum_{i=s}^{m-r+2} (i-1) \binom{m-i}{r-2} &\leq (m-r+1) \cdot \sum_{i=s}^{m-r+2} \binom{m-i}{r-2} =
            (m-r+1) \cdot \sum_{j=r-2}^{m-s}\binom{j}{r-2} \\ &= (m-r+1) \cdot \binom{m-s+1}{r-1} 
            \leq (m-r+1) \cdot \binom{\frac{m-r+1}{2}}{r-1} \\ &\leq
            \frac{(m-r+1)^r}{2^{r-1} (r-1)!} \leq \frac{(m-r+1)^r}{2 r!} < \frac{1}{2}\binom{m}{r},
        \end{align*}
        where the penultimate inequality uses $2^{r-2} \geq r$, which holds for $r \geq 4$. 
        As $\sum_{i=1}^{m-r+2} (i-1) \binom{m-i}{r-2} = \binom{m}{r}$, it follows that 
        $\sum_{i=2}^{s-1} (i-1) \binom{m-i}{r-2} > \frac{1}{2}\binom{m}{r} \geq f$. Now, by the definition of the numbers $d_i$, it is impossible that $d_i = i-1$ for all $2 \leq i \leq s-1$. Hence, there is $2 \leq i \leq s-1 \leq \frac{m+r}{2}$ with $d_{i} \leq i-2$. This means that $i^* \leq \frac{m+r}{2}$, proving Item (a).

        Next we prove Item (b) by showing that for every $i$, if $d_{i-1} \leq i-3$ then $d_i < \frac{r-2}{m-r+3-i} + 1$ and $d_i \leq i-2$. This would imply Item (b) by induction on $i$, because $d_{i^*} \leq i^*-2$. 
        So suppose that $d_{i-1} \leq i-3$. 
        As mentioned above, this means that 
        $\sum_{h=1}^{i-1} d_h w_h > f - w_{i-1}$. Combining this with 
        $\sum_{h=1}^{i} d_h w_h \leq f$, we find that 
        $$
        d_i < \frac{w_{i-1}}{w_i} = 
        \frac{\binom{m-i+1}{r-2}}{\binom{m-i}{r-2}} = \frac{m-i+1}{m-i-r+3} = \frac{r-2}{m-r+3-i} + 1.
        $$ 
        Now let us prove that $d_i \leq i-2$.
        Plugging $i \leq m-r+2$ into the above, we get 
        $d_i < r-1$, so we are done if $i \geq r$. And if $i \leq r$ then by the above we have $d_i < \frac{r-2}{m-2r+3} + 1 < 2$, so $d_i \leq 1 \leq i-2$.
        Finally, for $i^* < i \leq m-2r+5$ we get $d_i < \frac{r-2}{m-r+3-i}+1 \leq 2$, and hence $d_i \leq 1$. This proves (b).
        
        We now prove (c). By Item (b), we have $d_i \leq i-2$ for $i = m-r+2$. As we saw before, this implies that 
        $$
        \sum_{h=1}^{m-r+2} d_h w_h > f-w_{m-r+2} = f - \binom{m-(m-r+2)}{r-2} = f - 1.
        $$
        On the other hand,
        $\sum_{h=1}^{m-r+2} d_h w_h \leq f$ by the definition of the numbers $d_i$. This proves (c).

        Finally we prove (d). 
        Put $s := \lceil 2(r-2) \log(2(r-2)) \rceil$.
        Suppose first that there is at most one index $i^* < i \leq i^* + 2s$ with $d_i \geq 1$. 
        Then there are $s \geq 2(r-2) \log(2(r-2))$ such consecutive indices $i$ with $d_i = 0$. Note that 
        $$
        i^* + 2s \leq \frac{m+r}{2} + 4(r-2) \log(2(r-2)) + 2 \leq m-2r+5,
        $$
        where the first inequality follows from Item (a) and the second inequality holds for $m \geq 5r^2$ and $r \geq 3$, so we are done. 
        Suppose now that there are at least two indices $i^* < i \leq i^* + 2s$ with $d_i \geq 1$. Let $k < \ell$ be the first two such indices. Then $d_i = 0$ for every $i \in \{k+1,\dots,\ell-1\}$, so we only need to show that $\ell-k \geq 2(r-2) \log(2(r-2))$.
        By Item (b), we have $d_k,d_{\ell} = 1$. Also, as $k-1 \geq i^*$, we have $d_{k-1} \leq (k-1) - 2$ by Item (b). This implies that $f-\sum_{h=1}^{k-1} d_h w_h < w_{k-1}$. On the other hand, 
        $\sum_{h=1}^{m-r+2} d_h w_h = f$, so 
        $$
        w_{\ell} \leq \sum_{h=k+1}^{m-r+2} d_h w_h = 
        f-\sum_{h=1}^{k} d_h w_h = f - \sum_{h=1}^{k-1} d_h w_h - w_k,
        $$
        using that $d_k = d_{\ell} = 1$. Combining these two inequalities, we get that $w_{\ell} < w_{k-1} - w_k$. Recalling that $w_j = \binom{m-j}{r-2}$, we get
        $$
        \binom{m-\ell}{r-2} < \binom{m-k+1}{r-2} - \binom{m-k}{r-2} = 
        \frac{r-2}{m-k-r+3} \binom{m-k}{r-2},
        $$
        so 
        \begin{align*}
        \frac{m-k-r+3}{r-2} &< \frac{\binom{m-k}{r-2}}{\binom{m-\ell}{r-2}} =
        \prod_{i=0}^{r-3} \frac{m-k-i}{m-\ell-i} 
        \leq \left(\frac{m-k-r+3}{m-\ell-r+3} \right)^{r-2} \\ &= 
        \left( 1 + \frac{\ell-k}{m-\ell-r+3} \right)^{r-2} \leq 
        e^{\frac{(r-2)(\ell-k)}{m-\ell-r+3}}.
        \end{align*}
        Therefore,
        $$\ell-k > \frac{m-\ell-r+3}{r-2} \cdot \log\left(\frac{m-k-r+3}{r-2}\right) \geq 2(r-2) \log(2(r-2)),
        $$
        as required. Here, we used that 
        \begin{align*}
            m-r+3-k &\geq m-r+3-\ell \geq m-r+3-i^* - 2s \\
            &\geq m-r+3 - \frac{m+r}{2} - 4(r-2) \log(2(r-2)) - 2 \geq 2 (r-2)^2,
        \end{align*}
        where the last inequality holds for $m \geq 5 r^2$ and $r \geq 3$.
    \end{proof}
    
    

    We now define the graph $H$ on the vertices $u_1,\dots,u_{m-r+2}$ in such a way that the backward-degree of $u_i$ is $d_i$ for every $1 \leq i \leq m-r+2$. This will ensure that $H$ has total weight $f$, because $w_U(H) = \sum_{i=1}^{m-r+2} d_i w_i = f$ by Claim \ref{d_i}(c).
    The vertices $u_1,\dots,u_{i^*-1}$ form a clique (this corresponds to the definition of $i^*$, as $d_i = i-1$ for every $i < i^*$). The vertex $u_{i^*}$ is connected to the vertices $u_i$ with $1 \leq i \leq d_{i^*}$. 
    Next consider the vertices $u_i$ with $i^* < i \leq m-2r+5$, which have $d_i \leq 1$ by Claim \ref{d_i}(b). 
    If $d_i = 0$ then we do not need to add any backward edges from $u_i$. 
    Suppose that $d_i = 1$.
    Let $j = j(i)$ be the largest $j < i$ with $d_j = 0$, and connect $u_i$ to $u_j$. This is well-defined because $d_1 = 0$. Also, note that by the definition of $i^*$, we have either $j(i) = 1$ or $j(i) \geq i^*$.
    Observe that the subgraph induced by $u_{i^*},u_{i^*+1},\dots,u_{m-2r+5}$ is a monotone star forest (defined in Section \ref{subsec:prelim}). Indeed, the leaves of the stars are the vertices $u_i$ with $d_i = 1$, and the centers are the vertices $u_i$ with $d_i = 0$ (a star may have no leaves). 

    Finally, it remains to define the backward-edges of the vertices $u_i$ with $m-2r+6 \leq i \leq m-r+2$. By Claim \ref{d_i}(b) we have $d_i < \frac{r-2}{m-r+3-i} + 1$, and in particular $d_{m-r+2} < r-1$, so $d_{m-r+2} \leq r-2$. 
    Hence, the sum of backward degrees of $u_{m-2r+6},\dots,u_{m-r+2}$ is 
    $$
    D := \sum_{i=m-2r+6}^{m-r+2}\!\!\!d_i \leq \sum_{i=m-2r+6}^{m-r+2} \left( \frac{r-2}{m-r+3-i} + 1 \right) - 1 = 
    \sum_{j = 1}^{r-3}\left( \frac{r-2}{j} + 1\right) - 1
    \leq 2(r-2)\log(2(r-2))-2.
    $$
    By Claim \ref{d_i}(d), there are $i^* \leq k < \ell \leq m-2r+5$ with $\ell-k \geq D+2$ such that $d_i = 0$ for every $k+1 \leq i \leq \ell-1$. Note that in particular $d_{k+D+1} = 0$. 
    For each $m-2r+6 \leq i \leq m-r+2$, we connect $u_i$ to $d_i$ consecutive vertices in $\{u_{k+1},\dots,u_{k+D}\}$, such that for every $m-2r+6 \leq i < j \leq m-r+2$, all neighbours of $u_i$ come after all neighbours of $u_j$. 
    Note that the edges between 
    $\{u_{m-2r+6},\dots,u_{m-r+2}\}$ and 
    $\{u_{k+1},\dots,u_{k+D}\}$ form a nested-star forest (defined in Section \ref{subsec:prelim}). 
    
    This completes the definition of $H$, see Figure \ref{figure-H}. It remains to prove the following:

\begin{figure}[h]
\centering
\begin{tikzpicture}[thick,
  bnode/.style={draw,circle,fill,inner sep=1pt},every fit/.style={ellipse,draw,inner sep=-2pt,minimum height=0.4cm}
]

\node[bnode,xshift=0.75*1cm, label=below:$u_1$] (1) {};
\node[bnode,xshift=0.75*2cm, label=below:$u_{d_{i^*}}$] (2) {};
\node[bnode,xshift=0.75*3cm] (3) {};
\node[bnode,xshift=0.75*4cm, label=below:$u_{i^*-1}$] (4) {};
\node[bnode,xshift=0.75*5cm, label=below:$u_{i^*}$] (5) {};
\path (1) -- node[auto=false]{\ldots} (2);
\path (3) -- node[auto=false]{\ldots} (4);
\node[fit=(1) (4)] {};
\node[bnode,xshift=0.75*6cm] (6) {};
\path (5) -- node[auto=false]{\ldots} (6);
\node[bnode,xshift=0.75*7cm] (7) {};
\path (7) -- node[auto=false]{\ldots} (8);
\node[bnode,xshift=0.75*8cm] (8) {};
\node[bnode,xshift=0.75*9cm, label=below:$u_{k}$] (9) {};
\path (8) -- node[auto=false]{\ldots} (9);
\node[bnode,xshift=0.75*10cm] (10) {};
\node[bnode,xshift=0.75*11cm] (11) {};
\path (10) -- node[auto=false]{\ldots} (11);
\node[bnode,xshift=0.75*12cm] (12) {};
\node[bnode,xshift=0.75*13cm] (13) {};
\path (11) -- node[auto=false]{\ldots} (12);
\path (12) -- node[auto=false]{\ldots} (13);
\node[bnode,xshift=0.75*14cm, label=below:$u_{k+D+1}$] (14) {};
\node[bnode,xshift=0.75*15cm] (15) {};
\path (14) -- node[auto=false]{\ldots} (15);
\node[bnode,xshift=0.75*16cm] (16) {};
\node[bnode,xshift=0.75*17cm] (17) {};
\path (16) -- node[auto=false]{\ldots} (17);
\node[bnode,xshift=0.75*19cm, label=below:$u_{m-2r+6}$] (21) {};
\node[bnode,xshift=0.75*21cm, label=below:$u_{m-r+2}$] (22) {};
\path (17) -- node[auto=false]{\ldots} (21);
\path (21) -- node[auto=false]{\ldots} (22);

\draw [decorate, decoration = {brace,amplitude=10pt}] (3.45,-0.5) -- (0.3,-0.5) ;
\draw (1.9,-1) node {\footnotesize clique};

\draw (1) to [out=90,in=90,looseness=1] (5);
\draw (2) to [out=90,in=90,looseness=1] (5);
\draw (6) to [out=90,in=90,looseness=1] (7);
\draw (6) to [out=90,in=90,looseness=1] (8);
\draw (15) to [out=90,in=90,looseness=1] (16);
\draw (15) to [out=90,in=90,looseness=1] (17);

\draw (10) to [out=90,in=90,looseness=1] (22);
\draw (11) to [out=90,in=90,looseness=1] (22);
\draw (12) to [out=90,in=90,looseness=1] (21);
\draw (13) to [out=90,in=90,looseness=1] (21);

\end{tikzpicture}
\caption{The graph $H$ from Lemma \ref{r-graphs induction base case}}\label{figure-H}
\end{figure}
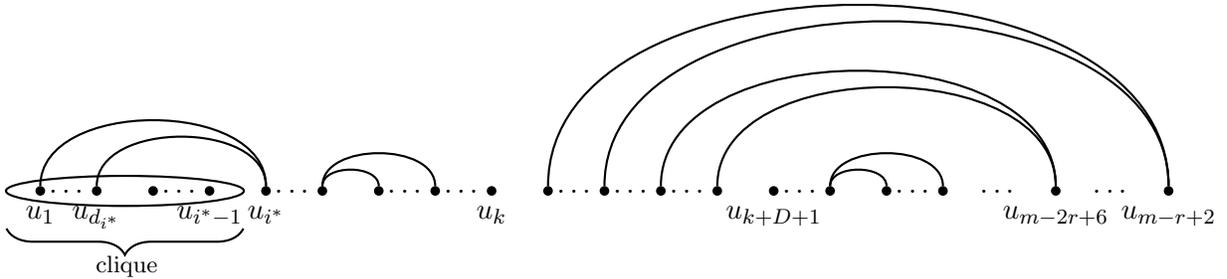

\begin{claim}\label{building-H}
    $H$ has the ordered EH-property. 
\end{claim}
\begin{proof}
        We will repeatedly use Lemma \ref{lem:substitution}, stating that the ordered EH-property is closed under substitution, and Lemma \ref{lem:star forests EH}, stating that all monotone star forests and nested star forests have the ordered EH-property. 
        Let $i_1$ be the smallest $i \geq i^*$ with $d_i = 0$; this is well-defined because $d_{k+1} = 0$. Note that $j(i) = 1$ for every $i^* < i < i_1$ and $j(i) \geq i_1$ for every $i > i_1$.
        Consider the three intervals $V_1 := \{u_1,\dots,u_{i_1-1}\}$, 
        $V_2 := \{u_{i_1},\dots,u_{k}\}$, $V_3 := \{u_{k+1},\dots,u_{m-r+2}\}$. We claim that there are no edges between these three sets. 
        Indeed, $j(i) \geq i_1$ for every $i > i_1$, so $V_2$ sends no edges to $V_1$. 
        Also, the vertices $u_{m-2r+6},\dots,u_{m-r+2}$ only send edges inside $V_3$, and for every $k+1 \leq i \leq m-2r+5$ with $d_i = 1$ we have $j(i) \geq k+1$ (because $d_{k+1} = 0$), implying that $V_3$ does not send edges to $V_1,V_2$. 
        So it suffices to show that the three graphs $H[V_1],H[V_2],H[V_3]$ have the ordered EH-property. 
        Recall that $H[V_2]$ is a monotone star forest, so this is true for $H[V_2]$.

        Next we consider $H[V_1]$. 
        First, if $i_1 = i^*$ then $H[V_1]$ is the clique $\{u_1,\dots,u_{i^*-1}\}$, so we are done. Suppose now that $i_1 > i^*$. In particular, $d_{i^*} \geq 1$, so $u_{i^*}$ is adjacent to $u_1$. 
        Then $H[V_1]$ consists of the clique 
        $\{u_1,\dots,u_{i^*-1}\}$, the edges between $u_{i^*}$ and $u_i$ for each $1 \leq i \leq d_{i^*}$, and the edges between $u_i$ and $u_1$ for each $i^* < i < i_1$. 
        Let $S_2$ denote the monotone star with vertices $1,2,3$ and edges $12,13$.
        The subgraph $H[\{u_2,\dots,u_{i^*}\}]$ can be obtained from $S_2$  
        by substituting the clique $\{u_2,\dots,u_{d_{i^*}}\}$ into vertex $1$, the clique $\{u_{d_{i^*}+1},\dots,u_{i^*-1}\}$ into vertex $2$, and $u_{i^*}$ into vertex $3$. Now, $H[V_1]$ can be obtained from $S_2$ by substituting $u_1$ into vertex $1$, $H[\{u_2,\dots,u_{i^*}\}]$ into vertex $2$, and the independent set $\{u_{i^*+1},\dots,u_{i_1-1}\}$ into vertex $3$. This shows that $H[V_1]$ has the ordered EH-property.  

        Finally, we consider $H[V_3]$. Partition $V_3$ further into the three intervals $U_1 := \{u_{k+1},\dots,u_{k+D}\}$, 
        $U_2 := \{u_{k+D+1},\dots,u_{m-2r+5}\}$, 
        $U_3 := \{u_{m-2r+6},\dots,u_{m-r+2}\}$. Recall that $d_{k+D+1} = 0$, so $j(i) \geq k+D+1$ for each $k+D+1 < i \leq m-2r+5$, meaning that there are no edges between $U_1$ and $U_2$. 
        The graph $H[U_2]$ is a monotone star forest, hence it has the ordered EH-property.
        Also, $U_1$ is an independent set (because $d_i = 0$ for all $k+1 \leq i \leq k+D$). 
        Moreover, all edges touching $U_3$ go to $U_1$, and these edges form a nested star forest. Thus, if we contract the set $U_2$ into a single vertex, we get a nested star forest (i.e., the star forest formed by $U_1,U_3$ with one additional vertex ``in the middle"), which has the ordered EH-property. If we substitute $H[U_2]$ back into this contraction vertex, we get $H[V_3] = H[U_1 \cup U_2 \cup U_3]$. This shows that $H[V_3]$ has the ordered EH-property, finishing the \nolinebreak proof. 
\end{proof}
\noindent
Claim \ref{building-H} completes the proof of Lemma \ref{r-graphs induction base case}.
\end{proof}

\subsection{Proof of Proposition \ref{thm:m=r+1}}\label{subsec:r+1}
\begin{proof}[Proof of Proposition \ref{thm:m=r+1}]
Let $0 \leq f \leq r+1$. 
By taking complements, we may assume that $f \leq \lfloor \frac{r+1}{2} \rfloor \leq r-1$ (using $r \geq 3$). 
Let $G$ be an $n$-vertex $r$-graph with no $(r+1,f)$-subset. 
We identify $G$ with the corresponding coloring $c:E(K_n^{(r)}) \to \{0,1\}$. 
Set $k := f$ if $f \geq 1$ and $k := 1$ if $f=0$.
Putting $\ell := \big\lceil \frac{1}{2} \log^{\frac{1}{2}}_{(r-2)}(n) \big\rceil$, apply Theorem \ref{erdos-rado} to $G$ with the above $k$ (note that $1 \leq k \leq r-1$) to obtain vertices $v_1,\dots,v_{\ell} \in V(G)$ and colors $\chi(v_i,v_j)$ such that $c(v_{i_1},\dots,v_{i_r}) = \chi(v_{i_{k}},v_{i_{k+1}})$ for every $1 \leq i_1 < \dots < i_r \leq \ell$. 
Let $G'$ be the ordered graph on $V(G') = \{v_1,\dots,v_{\ell}\}$ with edge-set $\{v_iv_j : \chi(v_iv_j) = 1\}$.

Suppose first that $f \geq 1$, so $k = f$. Note that if there are three indices $i,j,h$ with $f \leq i < j < h \leq \ell-r+f+1$ such that $v_iv_j,v_iv_h \notin E(G')$ and $v_jv_h \in E(G')$, then the $(r+1)$-set 
$U = \{v_1,\dots,v_{f-1},v_i,v_j,v_h,v_{\ell-r+f+2},\dots,v_{\ell}\}$ spans exactly $f$ edges. This is because every edge in $G[U]$ has to contain $v_j,v_h,v_{\ell-r+f+2},\dots,v_{\ell}$ and exactly $f-1$ of the $f$ vertices $v_1,\dots,v_{f-1},v_i$ (in other words, the edge $v_jv_h$ has weight $f$). Since $G$ has no $(r+1,f)$-subset, 
$G'[\{v_f,\dots,v_{\ell-r+f+1}\}]$ has no induced copy of the ordered graph with vertices $1,2,3$ and single edge $23$. Since this graph has the ordered EH property, $G'$ contains a homogeneous set of size at least $(\ell-r+2)^{a}$ for some absolute constant 
$a > 0$. This set is also homogeneous in $G$, and has size at least 
$(\ell-r+2)^a \geq 
\big(\log_{(r-2)}(n) \big)^{a/3}$, say. 

The case $f = 0$ is similar: If 
$G'[\{v_1,\dots,v_{\ell-r+2}\}]$ has an independent set of size $3$ then this gives an $(r+1,0)$-subset in $G$, and otherwise 
$G'$, and hence also $G$, has a clique of size at least 
$(\ell-r+2)^a \geq 
\big(\log_{(r-2)}(n) \big)^{a/3}$.
This completes the proof. 
\end{proof}

\section{Proof of Theorem \ref{main-result-2}}\label{sec:subgraph-densities}

We first give an overview of the proof of Theorem \ref{main-result-2}.
The main idea is to find disjoint vertex-sets $A_1,\dots,A_t$, each of size $m$, say, such that all triples of sets $(A_i,A_j,A_k)$ are homogeneous, i.e., complete or empty. Here $i,j,k$ are not necessarily distinct, so we also consider triples of the form $(A_i,A_i,A_j)$ and $(A_i,A_i,A_i)$. Note that if we select $x_i$ vertices from $A_i$ for each $1 \leq i \leq t$, then the number of edges spanned by the selected vertices is a cubic polynomial in $x_1,\dots,x_t$. We require that $x_1 + \dots + x_t = m$, since we are considering $m$-vertex subgraphs of $G$. We will be able to show that in many cases, this cubic polynomial obtains $\Theta(m^3)$ values. To this end, we have to make sure that the polynomial is, in some sense, non-degenerate. For example, we want to make sure that $A_1 \cup \dots \cup A_t$ is not a clique or an independent set. A slightly less trivial example is the following: If it is possible to partition $[t]$ into two sets $I,J$ such that $A := \bigcup_{i \in I}A_i$ and $B := \bigcup_{i \in J}A_i$ are both homogeneous, and so are the edges intersecting both $A,B$, then the cubic polynomial attains only $O(m)$ values on inputs $(x_1,\dots,x_t)$ with $x_1 + \dots + x_t = m$, because the value of the polynomial depends only on the number of vertices taken from $A$, i.e., $\sum_{i \in I}x_i$. So we see that the configuration $A_1,\dots,A_t$ needs to be ``irreducible" in some sense. 
We will not define what this means, but directly find certain configurations where the corresponding cubic polynomials obtain $\Theta(m^3)$ values.  
This is done in Lemma \ref{main-lemma}, which is the main ingredient in the proof of Theorem \ref{main-result-2}. 
We note that in some cases we will find a sequence of pairs of sets $A_1,\dots,A_t,B_1,\dots,B_t$
(instead of a sequence of sets $A_1,\dots,A_t$), and a certain relationship between $A_i,B_i$ will ensure that a corresponding cubic polynomial obtains $\Theta(m^3)$ values. 

To handle the cubic polynomials arising in our proof, we will use the following two lemmas. We defer the proofs of these lemmas to Section \ref{sec:cubic forms}. 

\begin{lemma} \label{corollary-1}
    Let $a,b,c,d,e \in \mathbb{Q}$ which do \underline{not} satisfy that $a=b =\frac{c}{3}$ or that $a=-b,c=0$. 
    Then the function 
	%
    $$
	f(x_1,\dots,x_m) = a \sum_{1 \leq i < j \leq m} x_ix_j^2 + 
    b \sum_{1 \leq i < j \leq m} x_i^2x_j + c \sum_{1 \leq i < j < k \leq m} x_i x_j x_k + d \sum_{1 \leq i \leq m} x_i^2 + e \sum_{1 \leq i < j \leq m} x_i x_j
	$$
    takes $\Omega(m^3)$ distinct values on the set of integer vectors 
    $(x_1,\dots,x_m)$ with $x_1,\dots,x_m \geq 0$ and $\sum_{i=1}^mx_i = m$.
\end{lemma}


\begin{lemma} \label{additional-lemma}
    The function
	$$
	f(a_1,\dots,a_m,b_1,\dots,b_m) = (a_1 + \dots + a_m) \sum_{1 \leq i < j \leq m} b_i b_j + (b_1 + \dots + b_m) \sum_{1 \leq i < j \leq m} a_i a_j
	$$
    takes $\Omega(m^3)$ distinct values on the set of integer vectors $(a_1,\dots,a_m,b_1,\dots,b_m)$ with $a_i,b_i \geq 0$ and 
    $\sum_{i=1}^m (a_i+ \nolinebreak b_i) = m$.
\end{lemma}

Lemma \ref{additional-lemma} is needed because in one of the cases in the proof of Theorem \ref{main-result-2}, we arrive at a function $f$ as in Lemma \ref{corollary-1}, but with $a=b=\frac{c}{3}$, which is not allowed by Lemma \ref{corollary-1}. This case is handled by Lemma \ref{additional-lemma}.

We now state some definitions and preliminary lemmas. In what follows, we will assume that the number of vertices $n$ is large enough where needed. When we say that a $3$-graph is large enough, we mean its number of vertices. 

\begin{definition}
    In a 3-graph $G$, we define the {\em density} $d_G(X,Y,Z)$ between (not necessarily disjoint) $X,Y,Z \subseteq V(G)$ as the fraction of triples $\{x,y,z\}$, $x \in X,y \in Y, z \in Z$, which form edges of $G$, i.e.,
    $$d_G(X,Y,Z) = \frac{|\{xyz \in E(G) :  x \in X, y \in Y, z \in Z\}|}{|\{xyz \in \binom{V(G)}{3} : x \in X, y \in Y, z \in Z\}|}.$$ 
    In particular, for disjoint $X,Y,Z \subseteq V(G)$ we have 
    $d_G(X,Y,Z) = \frac{e(X,Y,Z)}{|X||Y||Z|}$, where $e(X,Y,Z)$ is the number of edges intersecting each of $X,Y,Z$. For disjoint sets $X,Y$, we have
    $d_G(X,X,Y) = \frac{e(X,X,Y)}{\binom{|X|}{2}|Y|}$, where $e(X,X,Y)$ is the number of edges with two vertices in $X$ and one in $Y$. Similarly, $d_G(X,X,X) = \frac{e(X)}{\binom{|X|}{3}}$, where $e(X)$ is the number of edges inside $X$. 
\end{definition}

\begin{definition}
    For $v \in V(G)$ and $S \subseteq V(G) \setminus \{v\}$, we say that $(v,S)$ is a {\em star} of size $m$ in $G$ if $d(\{v\},S,S)=1$ and $|S| = m$. We say that $(v,S)$ is an {\em induced star} if additionally $d(S,S,S)=0$. An (induced) {\em antistar} in $G$ is an (induced) star in the complement of $G$.
\end{definition}

For a vertex $v$ in a $3$-graph $G$, recall that the {\em link graph} $L(v)$ is the graph with vertex-set $V(G)\setminus\{v\}$ and edge-set $\{xy : xyv \in E(G)\}$.
Note that $(v,S)$ is a star (resp. antistar) if and only if $S$ is a clique (resp. independent set) in $L(v)$. 

In the following lemma we apply Ramsey's theorem to show that given a collection of disjoint vertex-sets of very large (but constant) size, we can pass to large subsets of these sets such that all triples of subsets are homogeneous. 



\begin{lemma} \label{0-1-densities}
    For every $\ell, P \geq 0$, there exists $Q=Q(\ell,P)$ with the following property: Let $A'_1, \dots, A'_{\ell}$ be disjoint vertex-sets in a $3$-graph with $|A'_i| \geq Q$ for every $i$. Then there exist subsets $A_i \subseteq A'_i$, $i=1,\dots,\ell$, with $|A_i|=P$, such that $d(A_i,A_j,A_k) \in \{0,1\}$ for all (not necessarily distinct) $i,j,k \in [\ell]$.
\end{lemma}

\begin{proof}
    The proof is by applying the following claim repeatedly:
    \begin{claim}
       For every $q \geq 0$ there is $R = R(q)$ such that the following holds:
       \begin{enumerate}
           \item[(a)] If $X'$ is a vertex-set with $|X'| \geq R$, then there is $X \subseteq X'$ with $|X| = q$ and $d(X,X,X) \in \{0,1\}$. 
           \item[(b)] If $X',Y'$ are disjoint vertex-sets with $|X'|,|Y'| \geq R$, then there are $X \subseteq X', Y \subseteq Y'$ with $|X| = |Y| = q$ and $d(X,X,Y) \in \{0,1\}$. 
           \item[(c)] If $X',Y',Z'$ are disjoint vertex-sets with $|X'|,|Y'|,|Z'| \geq R$, then there are $X \subseteq X', Y \subseteq Y', Z \subseteq Z'$ with $|X| = |Y| = |Z| = q$ and $d(X,Y,Z) \in \{0,1\}$.
       \end{enumerate}
    \end{claim}
    \begin{proof}
        The first item is just Ramsey's theorem for 3-graphs. To prove the second item, fix $2q$ vertices $y_1,\dots,y_{2q} \in Y'$ and color each pair $x_1x_2 \in \binom{X'}{2}$ using one of $2^{2q}$ colors, according to the relation of $x_1x_2$ with $y_1,\dots,y_{2q}$; i.e., the color of $x_1x_2$ is the set of indices $i \in [2q]$ such that $x_1x_2y_i$ is an edge. Then apply the multicolor (graph) Ramsey theorem to this coloring to obtain $X \subseteq X'$ such that for each $1 \leq i \leq 2q$, either $y_ix_1x_2$ is an edge for all $x_1,x_2 \in X$, or $y_ix_1x_2$ is a non-edge for all $x_1,x_2 \in X$. One of these options occurs for at least $q$ of the vertices $y_1,\dots,y_{2q}$, and we take $Y$ to be the set of these $q$ vertices.  

       The proof of the third item is similar: Fix $z_1,\dots,z_{2q} \in Z'$ and apply the multicolor Ramsey theorem for bipartite graphs to $X' \times Y'$, where a pair $xy \in X' \times Y'$ is colored in one of $2^{2q}$ colors according to its relation to $z_1,\dots,z_{2q}$. Ramsey's theorem gives a monochromatic $q \times q$ bipartite graph $X \times Y$. Then take $Z$ to be a set of $q$ vertices among $z_1,\dots,z_{2q}$ having the same relation \nolinebreak to \nolinebreak $X \times Y$.
    \end{proof}
    Now, by applying Item (a) of the claim with $X' = A'_i$ for each $1 \leq i \leq \ell$, Item (b) with $X' = A'_i, Y' = A'_j$ for each $1 \leq i \neq j \leq \ell$, and Item (c) with $X'=A'_i,Y'=A'_j,Z'=A'_k$ for each $1 \leq i < j < k \leq \ell$, we get the required sets $A_i \subseteq A'_i$, $i=1,\dots,\ell$. Here we choose $Q(\ell,P)$ large enough so that after all of these applications, the remaining sets still have size at least $P$. This proves Lemma \ref{0-1-densities}.
\end{proof}

Suppose that $A_1,\dots,A_{\ell}$ are disjoint vertex-sets such that every triple $A_i,A_j,A_k$ (for not necessarily distinct $i,j,k$) is homogeneous. We then want to make sure that for each ``type" of triple $i,j,k$, all triples of this type are homogeneous in the same way, i.e., all are complete or all are empty. There are four types of triples: triples $(i,i,i)$; triples $(i,i,j)$ with $i < j$; triples $(i,i,j)$ with $i > j$; and triples $i,j,k$ with distinct $i,j,k$. By applying Ramsey's theorem for graphs and 3-graphs, we can indeed make sure that all triples of each type are homogeneous in the same way:

\begin{lemma} \label{0-1-densities-by-edgetype}
    For every $m \geq 1$ there is $\ell \geq 1$ with the following property: Let 
    $A_1, \dots, A_{\ell}$ be disjoint vertex-sets in a $3$-graph, and suppose that $d(A_i,A_j,A_k) \in \{0,1\}$ for all (not necessarily distinct) $i,j,k \in [\ell]$. Then there exists $I \subseteq [\ell]$ with $|I|=m$, and there exist $a,b,c,d \in \{0,1\}$, such that for all (not necessarily distinct) $i,j,k \in I$, 
    $$d(A_i,A_j,A_k) = \begin{cases}
    a &\text{if } i<j=k,\\
    b &\text{if } i=j<k,\\
    c &\text{if } i<j<k,\\
    d &\text{if } i=j=k.
    \end{cases}$$
\end{lemma}
\begin{proof}
Let $R_r(m;q)$ denote the $r$-uniform $q$-color Ramsey number of $m$, i.e., the minimum $n$ such that every $q$-coloring of $K_n^{(r)}$ has a monochromatic clique of size $m$. 
Then $\ell := 2 R_2(R_3(m;2);4)$ satisfies the requirement in the lemma. Indeed, first choose $I_1 \subseteq [\ell]$, $|I_1| = \ell/2$, such that all sets $(A_i : i \in I_1)$ are independent sets or all sets $(A_i : i \in I_1)$ are cliques. Then, for each $i,j \in I_1$, color the pair $ij$ with the color $(d(A_i,A_i,A_j),d(A_j,A_j,A_i)) \in \{0,1\}^2$. Applying Ramsey's theorem to this 4-coloring gives a monochromatic clique $I_2 \subseteq I_1$, $|I_2| = R_3(m)$. Now, for each $i,j,k \in I_2$, color the triple $ijk$ with the color $d(A_i,A_j,A_k)$. By the 3-uniform Ramsey's theorem, this coloring has a monochromatic clique $I$ of size $m$, as required. 
\end{proof}

In some cases, instead of considering a sequence of sets $A_1,\dots,A_{\ell}$, we will consider a sequence of pairs of sets $A_1,\dots,A_{\ell},B_1,\dots,B_{\ell}$. Here too we will want to make sure that all triples of the same ``type" are homogeneous in the same way. This is done in Lemma \ref{0-1-densities-by-edgetype-extended} below. Now there are more types. For example, the type $AAB$ refers to all triples $(A_i,A_j,B_k)$ with $i < j < k$. Similarly, there are the types $AAA,ABA,ABB,BAA,BAB,BBA,BBB$. Finally, there are the triples of the form \linebreak $(A_i,B_i,A_j),(A_i,B_i,B_j)$ for $i < j$ and for $i > j$, giving four additional types. 

\begin{lemma} \label{0-1-densities-by-edgetype-extended}
    For every $m \geq 1$ there is $\ell \geq 1$ with the following property: Let $A_1, \dots, A_{\ell},B_1, \dots, B_{\ell}$ be disjoint vertex-sets in a $3$-graph such that $d(X,Y,Z) \in \{0,1\}$ for all $X,Y,Z \in \{A_1, \dots, A_{\ell},B_1, \dots, B_{\ell}\}$. Then there exists $I \subseteq [\ell]$ with $|I|=m$, and there exist $a_1,a_2,b_1,b_2,c_1,c_2,c_3,c_4,c_5,c_6,c_7,c_8 \in \{0,1\}$, such that $d(A_i,A_j,B_j)=a_1, d(B_i,A_j,B_j)=a_2, d(A_i,B_i,A_j)=b_1, d(A_i,B_i,B_j)=b_2$ for all $i < j \in I$, and $d(A_i,A_j,B_k)=c_1, d(A_i,B_j,A_k)=c_2, d(A_i,B_j,B_k)=c_3, d(B_i,A_j,A_k)=c_4, d(B_i,A_j,B_k)=c_5, d(B_i,B_j,A_k)=c_6, d(A_i,A_j,A_k)=c_7, d(B_i,B_j,B_k)=c_8$ for all $i < j < k \in I$.
\end{lemma}
\begin{proof}
    The proof is similar to Lemma \ref{0-1-densities-by-edgetype}, and $\ell = R_2(R_3(m;2^8);2^4)$ suffices. Indeed, first color each pair $ij$, $1 \leq i < j \leq \ell$, with the color 
    $(d(A_i,A_j,B_j), d(B_i,A_j,B_j), d(A_i,B_i,A_j), d(A_i,B_i,B_j)) \in \{0,1\}^4$ and apply Ramsey's theorem. Then color each triple $i<j<k$ with one of $2^8$ colors to describe all densities $d(X,Y,Z)$ for $X \in \{A_i,B_i\}, Y \in \{A_j,B_j\}, Z \in \{A_k,B_k\}$, and apply the 3-uniform Ramsey theorem.  
\end{proof}

\noindent
The following well-known lemma by Spencer extends Turan's theorem to hypergraphs. 

\begin{lemma}[\cite{Spencer}]\label{independence-vs-averagedegree}
        Every $k$-uniform $n$-vertex hypergraph with average degree $d$ contains an independent set of size at least $(1-\frac{1}{k})\frac{n}{d^{1/(k-1)}}$.
\end{lemma}


We will use the following supersaturation version of the Erd\H{o}s-Szekeres bound \cite{ES} for off-diagonal Ramsey numbers:
 
\begin{lemma}\label{lem:Ramsey supersaturation}
	Let $t,n,m \geq 1$ with $m^{t-1} \leq n$, and let $G$ be an $n$-vertex graph with no clique of size $m$. Then $G$ has at least $\frac{n^t}{m^{t(t-1)}}$ independent sets of size $t$.
\end{lemma}
\begin{proof}
	As $R(K_t,K_m) \leq \binom{m+t-2}{t-1} \leq m^{t-1}$ (by the Erd\H{o}s-Szekeres bound \cite{ES}), every $m^{t-1}$ vertices contain an independent set of size $t$. By double counting, the number of such independent sets \nolinebreak is \nolinebreak at \nolinebreak least
	$$
	\frac{\binom{n}{m^{t-1}}}{\binom{n-t}{m^{t-1}-t}} = \frac{\binom{n}{t}}{\binom{m^{t-1}}{t}} = \prod_{i=0}^{t-1} \frac{n-i}{m^{t-1}-i} \geq \left( \frac{n}{m^{t-1}} \right)^t = \frac{n^t}{m^{t(t-1)}}.
	$$
\end{proof}

\begin{lemma}\label{lem:lior}
	For every $t \geq 1$ and $\gamma > 0$, there is $\delta > 0$ such that the following holds: Let $G$ be an $n$-vertex graph, for $n$ large enough, with average degree at least $n^{1-\delta}$ and no clique of size at least $n^{\delta}$. Then $G$ contains at least $n^{2t-\gamma}$ induced $K_{t,t}$'s.
\end{lemma}
\begin{proof}
	Note that at least $\frac{1}{2} n^{1-\delta}$ vertices of $G$ have degree at least $\frac{1}{2} n^{1-\delta}$ (else the average degree would be less than $n^{1-\delta}$). 
    Let $V_0$ be the set of these vertices.
    Fixing $v \in V_0$, Lemma \ref{lem:Ramsey supersaturation} with $m=n^{\delta}$ implies that the graph $G[N(v)]$ has at least $\frac{d(v)^t}{n^{t(t-1)\delta}} \geq 2^{-t} n^{t-t^2\delta}$ independent sets of size $t$. Summing over all $v \in V_0$, we get that $G$ has at least $2^{-t-1} n^{t+1-(t^2+1)\delta}$ pairs $(v,A)$ such that $|A|=t$, $A$ is independent and $v$ is adjacent to all vertices in $A$.
  
    For each independent set $A \subseteq V(G)$ with $|A|=t$, let $V(A) \subseteq V(G)$ be the set of vertices $v$ adjacent to all vertices of $A$. By the above, we have 
    $\sum_{A \subseteq V(G), |A|=t} |V(A)| \geq 
    2^{-t-1} n^{t+1-(t^2+1)\delta}$. It follows that there are at least 
    $2^{-t-2} n^{t-(t^2+1)\delta}$ choices of $A$ for which $|V(A)| \geq 2^{-t-2} n^{1-(t^2+1)\delta}$ (else the sum of $|V(A)|$ over all $A$ would be smaller than indicated). For each such choice of $A$, by Lemma \ref{lem:Ramsey supersaturation} with $m=n^{\delta}$, the graph $G[V(A)]$ has at least 
    $$
    \frac{|V(A)|^t}{n^{t(t-1)\delta}} \geq
    \frac{(2^{-t-2} n^{1-(t^2+1)\delta})^t}{n^{t(t-1)\delta}} = 2^{-(t+2)t} n^{t-(t+1)t^2\delta}$$ independent sets $B \subseteq V(A)$ of size $|B|=t$.

    In total, there are at least $2^{-t-2} n^{t-(t^2+1)\delta} \cdot 2^{-(t+2)t} n^{t-(t+1)t^2\delta} \geq n^{2t-\gamma}$ choices of disjoint sets $A,B$ which form an induced copy of $K_{t,t}$. Here we assumed that $\delta$ is small enough as a function of $\gamma,t$. 
	\end{proof}

Recall that our goal is to find certain configurations consisting of sets $A_1,\dots,A_{\ell}$ or \linebreak $A_1,\dots,A_{\ell},B_1,\dots,B_{\ell}$. In the following lemma, we show that if each large set of vertices has many induced stars, then we can find a suitable configuration $A_1,\dots,A_{\ell}$.

\begin{lemma} \label{main-case-1}
    Let $\ell, s \geq 0$ and $\theta \in (0,1)$. Let $G$ be an $n$-vertex $3$-graph where every subset $U \subseteq V(G)$ of size $|U| \geq n^{(1-\theta)^{\ell}}$ has at least $|U|^{s+1-\theta}$ induced stars $(v,S)$ of size $|S| = s$. Then $G$ contains disjoint subsets $A_1,\dots,A_{\ell} \subseteq V(G)$ with $|A_i|=s$ such that $d(A_i,A_i,A_i)=0$ for $i \in [\ell]$ and $d(A_i,A_j,A_j)=1$ for $1\leq i < j \leq \ell$.
\end{lemma}
\begin{proof}
    We prove this by induction on $\ell$. For $\ell=0$, the statement is trivial. Let $\ell \geq 1$.
    By assumption, $G$ contains at least $n^{s+1-\theta}$ induced stars of size $s$. By the pigeonhole principle, there is a set $A_{\ell} \subseteq V(G)$, $|A_{\ell}| = s$, such that $(v,A_{\ell})$ is an induced star for at least $\frac{n^{s+1-\theta}}{\binom{n}{s}} \geq n^{1-\theta}$ different vertices $v \in V(G)$. Let $V' \subseteq V(G)$ with $|V'| \geq n^{1-\theta}$ be the set of those vertices. Note that $G[V']$ has the property that every subset $U \subseteq V'$ of size $|U| \geq |V'|^{(1-\theta)^{\ell-1}} \geq n^{(1-\theta)^\ell}$ has at least $|U|^{s+1-\theta}$ induced stars of size $s$.
    
    Thus we can apply the induction hypothesis for $\ell-1$ and find disjoint subsets $A_1,\dots,A_{\ell-1} \subseteq V'$ with $|A_i|=s$ such that $d(A_i,A_i,A_i)=0$ for $i \in [\ell-1]$ and $d(A_i,A_j,A_j)=1$ for $1\leq i < j \leq \ell-1$. As $(v,A_{\ell})$ is an induced star for each vertex $v \in A_1 \cup \dots \cup A_{\ell-1}$, we also have $d(A_{\ell},A_{\ell},A_{\ell})=0$ and $d(A_i,A_{\ell},A_{\ell})=1$ for $i \in [\ell-1]$, completing the induction step.
\end{proof}

If the condition of Lemma \ref{main-case-1} does not hold, then we would like to pass to a subset of vertices which contains no induced star of size $s$. This is done in the following lemma. 

\begin{lemma} \label{no-small-induced-stars}
    Let $s \geq 1$ and $\theta \in (0,1)$, and let $G$ be an $n$-vertex $3$-graph, for $n$ large enough, which has less than $n^{s+1-\theta}$ induced stars of size $s$. Then $G$ contains a subset $U \subseteq V(G)$ with $|U| \geq n^{\frac{\theta}{2s}}$ such that $G[U]$ does not contain any induced star of size $s$.
\end{lemma}
\begin{proof}
    Let $H$ be the $(s+1)$-graph on $V(G)$ whose edges are $\{v\} \cup S$ for every induced star $(v,S)$ of size $s$ in $G$. By assumption, the average degree of $H$ is at most $d(H) \leq \frac{(s+1) n^{s+1-\theta}}{n} = (s+1) n^{s-\theta}$. Thus by Lemma \ref{independence-vs-averagedegree} for $k=s+1$, there is an independent set $U \subseteq V(H)$ of size at least 
    $$|U| \geq \left(1-\frac{1}{s+1}\right)\frac{n}{((s+1) n^{s-\theta})^{\frac{1}{s}}} = \frac{s}{(s+1)^{1+\frac{1}{s}}} n^{\frac{\theta}{s}} \geq n^{\frac{\theta}{2s}},$$ 
    using that $n$ is large enough. By definition of $H$, there is no induced star of size $s$ in $G[U]$.
\end{proof}


In the following lemma, we show that if a $3$-graph has no induced stars or antistars, then we can pass to a subset $W \subseteq V(G)$ which contains no (not necessarily induced) star or no (not necessarily induced) antistar of size at least $|W|^{\delta}$, for some small $\delta$. 

\begin{lemma} \label{no-large-star}
    Let $s \geq 0$ and $\delta > 0$, and let $G$ be an $n$-vertex $3$-graph, with $n$ large enough, such that $G$ does not contain any induced star or antistar of size $s$. Then $G$ contains a subset $W \subseteq V(G)$ with $|W| \geq n^\delta$ such that there is no star of size at least $|W|^\delta$ in $G[W]$ or there is no antistar of size at least $|W|^\delta$ in $G[W]$.
\end{lemma}
\begin{proof}
    If $G$ does not contain any star of size at least $n^\delta$, then $W=V(G)$ suffices. Else, $G$ has a (not necessarily induced) star $(v,W)$ of size $|W| \geq n^\delta$. If there were an antistar $(u,U)$ of size $|U| \geq |W|^\delta$ in $G[W]$, then, by Ramsey's theorem, there would be a clique or an independent set $S$ of size $|S| = s$ in $G[U]$ (assuming that $n$ is large enough). However, if $S$ is a clique then $(u,S)$ is an induced antistar, and if $S$ is an independent set then $(v,S)$ is an induced star, in both cases giving a contradiction to the assumption that $G$ has no induced stars and no induced antistars. It follows that there is no antistar of size at least $|W|^\delta$ in $G[W]$, so $W$ satisfies the required properties.
\end{proof}

The next lemma shows that if a $3$-graph has no large stars and no large independent sets, then one can find in it a certain configuration which can then be used to prove Theorem \ref{main-result-2}.

\begin{lemma} \label{main-case-2}
    For every $\ell, t \geq 0$, there is $\delta=\delta(\ell,t) > 0$ such that the following holds. Let $G$ be an $n$-vertex $3$-graph, with $n$ large enough, which has no star of size at least $n^{\delta}$ and no independent set of size at least $n^{\delta}$. Then there are disjoint subsets $A_1, \dots, A_{\ell}, B_1, \dots, B_{\ell} \subseteq V(G)$ with $|A_i|=|B_i|=t$ such that $d(A_i,A_j,B_j) = d(B_i,A_j,B_j) = 1$ and $d(A_i,A_j,A_j) = d(A_i,B_j,B_j) = d(B_i,A_j,A_j) = d(B_i,B_j,B_j) = 0$ for $1 \leq i<j \leq \ell$.
\end{lemma}
\begin{proof}
    We prove the statement and choose $\delta = \delta(\ell,t)$ by induction on $\ell \geq 0$. 
    For $\ell=0$, the statement is trivial. Now let $\ell \geq 1$.
    By assumption, $G$ has no star $(v,S)$ of size $|S| \geq n^{\delta}$. This means that for every $v \in V(G)$, the link graph $L(v)$ has no clique of size at least $n^{\delta}$.
    
    By Lemma \ref{independence-vs-averagedegree} for $k=3$, the average degree of $G$, denoted $d(G)$, is at least $d(G) \geq (\frac{2n}{3\alpha(G)})^2 \geq \frac{4}{9} n^{2-2\delta}$. Therefore, there are at least $\frac{2}{9} n^{1-2\delta}$ vertices $v \in V(G)$ with degree at least 
    $\frac{2}{9} n^{2-2\delta}$ (else the average degree of $G$ would be less than $\frac{4}{9} n^{2-2\delta}$). Let 
    $V_0 \subseteq V(G)$ be the set of those vertices. For $v \in V_0$, $L(v)$ has average degree at least $\frac{2}{9} n^{1-2\delta} \geq n^{1-3\delta}$ (for $n$ large enough). Setting $\delta$ small enough (as a function of $t$), we can apply Lemma \ref{lem:lior} with (e.g.) $\gamma = \frac{1}{4}$ and find that for every $v \in V_0$, $L(v)$ contains at least $n^{2t-1/4}$ induced $K_{t,t}$'s.

    For every pair of disjoint $A, B \subseteq V(G)$ with $|A|=|B|=t$, let $V(A, B) \subseteq V_0$ be the set of vertices $v$ such that $A,B$ are the parts of an induced $K_{t,t}$ in $L(v)$. 
    By the above, we have $\sum_{A,B}|V(A,B)| \geq 
    |V_0| \cdot n^{2t-1/4} \geq \frac{2}{9}n^{1-2\delta} \cdot n^{2t-1/4} \geq n^{2t+1/2}$, provided that $\delta$ is small enough.  
    Hence, there are some disjoint $A, B \subseteq V(G)$ with $|A|=|B|=t$ such that $|V(A, B)| \geq n^{1/2}$.
    
    Set $V'=V(A, B)$, $A_{\ell} = A$ and $B_{\ell} = B$; so $d(V',A_{\ell},B_{\ell}) = 1$, $d(V',A_{\ell},A_{\ell}) = d(V',B_{\ell},B_{\ell}) = 0$, and $|V'| \geq n^{1/2}$. 
    Choosing $\delta = \delta(\ell,t) \leq \frac{1}{2}\delta(\ell-1,t)$, we get that 
    $G[V']$ has no star or independent set of size at least 
    $n^{\delta(\ell,t)} \leq |V'|^{\delta(\ell-1,t)}$.
    Thus we can apply the induction hypothesis for $\ell-1$ and find disjoint sets $A_1, \dots, A_{\ell-1}, B_1, \dots, B_{\ell-1} \subseteq V'$ with $|A_i|=|B_i|=t$ such that $d(A_i,A_j,B_j) = d(B_i,A_j,B_j) = 1$ and $d(A_i,A_j,A_j) = d(A_i,B_j,B_j) = d(B_i,A_j,A_j) = d(B_i,B_j,B_j) = 0$ for $1 \leq i<j \leq \ell-1$. As $d(V',A_{\ell},B_{\ell}) = 1$ and $d(V',A_{\ell},A_{\ell}) = d(V',B_{\ell},B_{\ell}) = 0$, we also have $d(A_i,A_{\ell},B_{\ell}) = d(B_i,A_{\ell},B_{\ell}) = 1$ and $d(A_i,A_{\ell},A_{\ell}) = d(A_i,B_{\ell},B_{\ell}) = d(B_i,A_{\ell},A_{\ell}) = d(B_i,B_{\ell},B_{\ell}) = 0$ for $1 \leq i< \ell$, completing the induction step.
\end{proof}
\noindent
We now combine all of the above to prove our main lemma. 

\begin{lemma} \label{main-lemma}
    For every $m \geq 3$ there is a constant $\varepsilon = \varepsilon(m) > 0$ such that the following holds. Let $G$ be an $n$-vertex $3$-graph, with $n$ large enough, which does not contain a homogeneous set of size at least $n^\varepsilon$. Then $G$ or $\bar{G}$ contains one of the following structures:
    \begin{enumerate}
        \item[(a)] disjoint subsets $A_1, \dots, A_m \subseteq V(G)$ with $|A_i| \geq m$ and $a,b,c,d \in \{0,1\}$, not all equal, such that for all (not necessarily distinct) $i,j,k \in [m]$, 
        $$d(A_i,A_j,A_k) = \begin{cases}
    a &\text{if } i<j=k,\\
    b &\text{if } i=j<k,\\
    c &\text{if } i<j<k,\\
    d &\text{if } i=j=k.
    \end{cases}$$
        \item[(b)] disjoint subsets $A_1, \dots, A_m, B_1, \dots, B_m \subseteq V(G)$ with $|A_i|,|B_i| \geq m$ and $b_1,b_2,c_1,c_2,c_3,c_4,c_5,c_6 \in \{0,1\}$, such that 
        \begin{itemize}
            \item $d(A_i,A_j,B_j)=1, d(B_i,A_j,B_j)=1, d(A_i,B_i,A_j)=b_1, d(A_i,B_i,B_j)=b_2$ for all $1 \leq i < j \leq m$;
            \item $d(A_i,A_j,B_k)=c_1, d(A_i,B_j,A_k)=c_2, d(A_i,B_j,B_k)=c_3, d(B_i,A_j,A_k)=c_4, d(B_i,A_j,B_k)=c_5, d(B_i,B_j,A_k)=c_6$ for all $1 \leq i < j < k \leq m$;
            \item $d(X,Y,Z) = 0$ for all $X,Y,Z \in \{A_1,\dots,A_m,B_1,\dots,B_m\}$ which are not pairwise distinct.  
            \item $d(X,Y,Z) = 0$ for all $X,Y,Z \in \{A_1,\dots,A_m\}$ and for all $X,Y,Z \in \{B_1,\dots,B_m\}$.
        \end{itemize}
    \end{enumerate}
\end{lemma}
\begin{proof}
    We fix $\ell \geq 0$ large enough to satisfy both Lemma \ref{0-1-densities-by-edgetype} and Lemma \ref{0-1-densities-by-edgetype-extended} for the given $m \geq 3$. Let $s := Q(\ell,m)$ and $t=Q(2\ell+2,s)$ be from Lemma \ref{0-1-densities}. Set also $\theta = \frac{1}{2}$ (we will apply Lemma \ref{main-case-1} with this $\theta$) and let 
    $\delta = \delta(\ell+1,t)$ be from Lemma \ref{main-case-2}. We prove the lemma with
    $$
    \varepsilon = \varepsilon(m) := \frac{\delta^2}{2^{2\ell+4} s^2}.
    $$

    If every subset $U \subseteq V(G)$ of size $|U| \geq n^{(1-\theta)^{\ell}}$ has at least $|U|^{s+1-\theta}$ induced stars $(v,S)$ of size $|S| = s$ then, by Lemma \ref{main-case-1}, $G$ contains disjoint subsets $A'_1,\dots,A'_{\ell} \subseteq V(G)$ with $|A'_i|=s$ such that $d(A'_i,A'_i,A'_i)=0$ for $i \in [\ell]$ and $d(A'_i,A'_j,A'_j)=1$ for $1\leq i < j \leq \ell$. Now, by applying Lemma \ref{0-1-densities} with $P=m$, we get subsets $A_i \subseteq A'_i$ of size $m$ such that $d(A_i,A_j,A_k) \in \{0,1\}$ for all (not necessarily distinct) $1 \leq i,j,k \leq \ell$. Then, applying Lemma \ref{0-1-densities-by-edgetype}, we get $I \subseteq [\ell]$ of size $|I|=m$ and $b,c \in \{0,1\}$ such that for all (not necessarily distinct) $i,j,k \in I$, it holds that
    $$d(A_i,A_j,A_k) = \begin{cases}
    1 &\text{if } i<j=k,\\
    b &\text{if } i=j<k,\\
    c &\text{if } i<j<k,\\
    0 &\text{if } i=j=k.
    \end{cases}$$
    Thus, Item (a) in the lemma holds. 

    Now suppose that there is a subset $U \subseteq V(G)$ of size $|U| \geq n^{(1-\theta)^{\ell}}$ which has less than $|U|^{s+1-\theta}$ induced stars $(v,S)$ of size $|S| = s$. By Lemma \ref{no-small-induced-stars}, $G[U]$ contains a subset $U' \subseteq U$ with $|U'| \geq |U|^{\frac{\theta}{2s}}$ such that $G[U']$ does not contain any induced star of size $s$. 
    
    We now apply the same strategy to the complement $\bar{G}[U']$ of $G[U']$, as follows. If for every subset $U'' \subseteq U'$ of size $|U''| \geq |U'|^{(1-\theta)^{\ell}}$, $G[U'']$ contains at least $|U''|^{s+1-\theta}$ induced antistars $(v,S)$ of size $|S| = s$ then, by Lemma \ref{main-case-1} (applied to the complement $\bar{G}[U']$), $G[U']$ contains disjoint subsets $A'_1,\dots,A'_{\ell} \subseteq V(G)$ with $|A'_i|=s$ such that $d(A'_i,A'_i,A'_i)=1$ for $i \in [\ell]$ and $d(A'_i,A'_j,A'_j)=0$ for $1\leq i < j \leq \ell$. Again, we can apply Lemma \ref{0-1-densities} and then Lemma \ref{0-1-densities-by-edgetype} to get an index-set $I \subseteq [\ell]$ of size $|I| = m$, subsets $A_i \subseteq A'_i$ with $|A_i|=m$, and $b,c \in \{0,1\}$, such that for all (not necessarily distinct) $i,j,k \in I$, 
    $$d(A_i,A_j,A_k) = \begin{cases}
    0 &\text{if } i<j=k,\\
    b &\text{if } i=j<k,\\
    c &\text{if } i<j<k,\\
    1 &\text{if } i=j=k.
    \end{cases}$$
    So again Item (a) in the lemma holds. 

    Else, there is a subset $U'' \subseteq U'$ of size $|U''| \geq |U'|^{(1-\theta)^{\ell}}$ such that $G[U'']$ has less than $|U''|^{s+1-\theta}$ induced antistars $(v,S)$ of size $s$. By Lemma \ref{no-small-induced-stars} (applied to the complement $\bar{G}[U'']$), $G[U'']$ contains a subset $U^* \subseteq U''$ with $|U^*| \geq |U''|^{\frac{\theta}{2s}}$ such that $G[U^*]$ does not contain any induced antistar of size $s$. As $U^* \subseteq U'$, $G[U^*]$ also does not contain any induced star of size $s$. By Lemma \ref{no-large-star}, there is a subset $W \subseteq U^*$ with $|W| \geq |U^*|^\delta$ such that in $G[W]$, there is no star of size at least $|W|^\delta$ or there is no antistar of size at least $|W|^\delta$. 
    We assume without loss of generality that $G[W]$ has no star of size at least $|W|^\delta$. 
    Note that  
    \begin{align*}
    |W| \geq |U^*|^\delta \geq |U''|^{\frac{\theta}{2s} \delta} \geq |U'|^{(1-\theta)^{\ell} \frac{\theta}{2s} \delta} \geq |U|^{\frac{\theta}{2s} (1-\theta)^{\ell} \frac{\theta}{2s} \delta} \geq n^{(1-\theta)^{\ell} \frac{\theta}{2s} (1-\theta)^{\ell} \frac{\theta}{2s} \delta} = n^{\frac{\delta}{2^{2\ell+4} s^2}} = n^{\varepsilon/\delta},
    \end{align*} 
    using that $\theta = \frac{1}{2}$ and our choice of $\varepsilon$.
    
    By assumption, $G$ has no homogeneous set of size at least $n^\varepsilon \leq |W|^{\delta}$. Hence, $G[W]$ has no independent set of size at least $|W|^{\delta}$. Recall that $\delta = \delta(\ell+1,t)$ was chosen via Lemma \ref{main-case-2}.
    Therefore, we can apply Lemma \ref{main-case-2} to $G[W]$ and obtain disjoint sets $A'_0, \dots, A'_{\ell}, B'_0, \dots, B'_{\ell} \subseteq V(G)$ with $|A'_i|=|B'_i|=t$ such that $d(A'_i,A'_j,B'_j) = d(B'_i,A'_j,B'_j) = 1$ and $d(A'_i,A'_j,A'_j) = d(A'_i,B'_j,B'_j) = d(B'_i,A'_j,A'_j) = d(B'_i,B'_j,B'_j) = 0$ for $0 \leq i<j \leq {\ell}$.

    As $t=Q(2\ell+2,s)$ was chosen via Lemma \ref{0-1-densities}, we can apply Lemma \ref{0-1-densities} (with $P = s$) and get disjoint subsets $A_i \subseteq A'_i, B_i \subseteq B'_i$ of size $|A_i|=|B_i|=s \geq m$ such that $d(X,Y,Z) \in \{0,1\}$ for all (not necessarily distinct) $X,Y,Z \in \{A_0,\dots,A_{\ell},B_0,\dots,B_{\ell}\}$. 
    
    Recall that there are no induced stars or antistars of size $s$ in $G[W]$. This implies that $d(A_i,A_i,A_i)=d(B_i,B_i,B_i)=0$ for all $1 \leq i \leq \ell$. Indeed, if $d(A_i,A_i,A_i) = 1$, then every vertex in $A_0$ makes an induced antistar of size $|A_i| = s$ with $A_i$, since $d(A_0,A_i,A_i) = 0$. This is impossible, so $d(A_i,A_i,A_i) = 0$, and similarly 
    $d(B_i,B_i,B_i) = 0$. 
    This in turn implies that $d(X,X,Y)=0$ for every $X, Y \in \{A_1,\dots,A_{\ell}, B_1, \dots, B_{\ell}\}$; indeed, we already showed that $d(X,X,X) = d(Y,Y,Y) = 0$, so if \linebreak $d(X,X,Y) = 1$ then we get an induced star of size $|X| = s$. So we see that if $X,Y,Z \in \linebreak
    \{A_1,\dots,A_{\ell},B_1,\dots,B_{\ell}\}$ are not pairwise distinct, then $d(X,Y,Z) = 0$.
    
    Now apply Lemma \ref{0-1-densities-by-edgetype-extended} to $A_1, \dots, A_{\ell}, B_1, \dots, B_{\ell}$ to get a subset of indices $I \subseteq \{1,\dots,\ell\}$ with $|I|=m$, without loss of generality $I=[m]$, and $a_1,a_2,b_1,b_2,c_1,c_2,c_3,c_4,c_5,c_6,c_7,c_8 \in \{0,1\}$, such that $d(A_i,A_j,B_j)=a_1, d(B_i,A_j,B_j)=a_2, d(A_i,B_i,A_j)=b_1, d(A_i,B_i,B_j)=b_2$ for $i < j \in I$, and $d(A_i,A_j,B_k)=c_1, d(A_i,B_j,A_k)=c_2, d(A_i,B_j,B_k)=c_3, d(B_i,A_j,A_k)=c_4, d(B_i,A_j,B_k)=c_5, d(B_i,B_j,A_k)=c_6, d(A_i,A_j,A_k)=c_7, d(B_i,B_j,B_k)=c_8$ for $1 \leq i < j < k \leq m$. 
    Note that $a_1=a_2=1$ by the choice of the sets $A'_i,B'_i$ via Lemma \ref{main-case-2}.
    If $c_7=c_8=0$ then we exactly get Item (b) in the lemma. So suppose that $c_7 = 1$ or $c_8 = 1$. These two cases are symmetrical, so assume without loss of generality that $c_7=1$. Then we have that for all (not necessarily distinct) $i,j,k \in [m]$, 
    $$d(A_i,A_j,A_k) = \begin{cases}
    0 &\text{if } i<j=k,\\
    0 &\text{if } i=j<k,\\
    1 &\text{if } i<j<k,\\
    0 &\text{if } i=j=k.
    \end{cases}$$
    implying that Item (a) in the lemma holds. This completes the proof.  
\end{proof}
\noindent
Finally, we combine Lemma \ref{main-lemma} with Lemmas \ref{corollary-1}-\ref{additional-lemma} to prove Theorem \ref{main-result-2}.
\begin{proof}[Proof of Theorem \ref{main-result-2}]
    Note that $s(\bar{G};m) = s(G;m)$, so we can always pass to the complement, if necessary. In particular, we may assume that the conclusion of Lemma \ref{main-lemma} holds in $G$. Also, note that if Item (a) of Lemma \ref{main-lemma} holds for $G$ then it also holds for $\bar{G}$.
    
    We consider the two cases of Lemma \ref{main-lemma} separately.
    Suppose first that Item (a) in Lemma \ref{main-lemma} holds, and let $A_1,\dots,A_m\subseteq V(G)$ and $a,b,c,d \in \{0,1\}$ be as in that item. By taking complements if necessary, we can assume without loss of generality that $d=0$. 
    We now apply Lemma \ref{corollary-1}. For a vector of integers $(x_1, \dots, x_m)$ with $x_i \geq 0$ and $x_1 + \dots + x_m = m$, let $X_i \subseteq A_i$ with $|X_i|=x_i$, and consider the number of edges induced by $X := X_1 \cup \dots \cup X_m$. Clearly $|X| = m$, and we have
	\begin{align}\label{eq:case 1}
    \nonumber e(G[X]) &= a \sum_{1 \leq i < j \leq m} x_i\binom{x_j}{2} + b \sum_{1 \leq i < j \leq m} \binom{x_i}{2} x_j + c \sum_{1 \leq i < j < k \leq m} x_i x_j x_k \\ &=
    \frac{a}{2} \sum_{1 \leq i < j \leq m} x_i x_j^2 + \frac{b}{2} \sum_{1 \leq i < j \leq m} x_i^2 x_j + c \sum_{1 \leq i < j < k \leq m} x_i x_j x_k - \frac{a+b}{2} \sum_{1 \leq i < j \leq m} x_i x_j.
	\end{align}
    As $a,b,c \in \{0,1\}$ and $1 \in \{a,b,c\}$ (because $a,b,c,d$ are not all equal), it is not possible to have 
    $\frac{a}{2} = \frac{b}{2} = \frac{c}{3}$ or 
    $\frac{a}{2}=-\frac{b}{2},c=0$. Hence, we can apply Lemma \ref{corollary-1} (with $\frac{a}{2}$ in place of $a$, $\frac{b}{2}$ in place of $b$, $d=0$ and $e = -\frac{a+b}{2}$), to infer that the function in \eqref{eq:case 1} takes $\Theta(m^3)$ values on the set of vectors $(x_1,\dots,x_m)$ with $x_i \geq 0$ and $x_1 + \dots + x_m = m$.
    This means that the set $\{e(G[X]) : X \subseteq V(G), |X| = m\}$ has size $\Theta(m^3)$, hence
    the conclusion of Theorem \ref{main-result-2} holds. 
    
    Suppose now that Item (b) in Lemma \ref{main-lemma} holds, and let $A_1,\dots,A_m,B_1,\dots,B_m \subseteq V(G)$ and $b_1,b_2,c_1,\dots,c_6 \in \{0,1\}$ be as in that item. We proceed similarly to the previous case.
    Set $a:=2$, $b:=b_1+b_2 \in \{0,1,2\}$, and $c:=c_1+\dots+c_6 \in \{0,\dots,6\}$.
    Suppose first that $a=b=\frac{c}{3}$ does {\bf not} hold (so that we can apply Lemma  \ref{corollary-1}). Write $m = 2t+\epsilon$, where $\epsilon \in \{0,1\}$. For a vector $(x_1, \dots, x_t)$ with $x_i \geq 0$ and $x_1 + \dots + x_t = t$, choose a set of $x_i$ vertices from $A_i$ and a set of $x_i$ vertices from $B_i$. Also, choose $\epsilon$ vertices from $A_m$, and let $X$ be the set of all chosen vertices, so that $|X| = \sum_{i=1}^t 2x_i + \epsilon = 2t+\epsilon = m$. By the guarantees of Item (b) in Lemma \ref{main-lemma}, we have
	\begin{align*}
	    e(G[X]) = 2 \! \! \! \! \sum_{1 \leq i < j \leq t} x_i x_j^2 + b  \! \! \! \! \sum_{1 \leq i < j \leq t} x_i^2 x_j + c  \! \! \! \! \! \! \sum_{1 \leq i < j < k \leq t} x_i x_j x_k + \epsilon b_1 \! \sum_{1 \leq i \leq t} x_i^2 + \epsilon (c_2+c_4+c_6)  \! \! \! \! \sum_{1 \leq i < j \leq t} x_i x_j.
	\end{align*} 
    By Lemma \ref{corollary-1}, the function on the right-hand side takes $\Theta(t^3) = \Theta(m^3)$ values on vectors $(x_1,\dots,x_t)$ with $x_i \geq 0$ and $x_1 + \dots + x_t = t$. This implies the statement of Theorem \ref{main-result-2}. 
    
    Finally, suppose that $a=b=\frac{c}{3}$. As $a=2$, this implies that $a=b=2$ and $c=6$, so $b_i = 1$ for $i=1,2$ and $c_i = 1$ for $i=1,\dots,6$. This means that for every three distinct sets $X,Y,Z \in \{A_1,\dots,A_m,B_1,\dots,B_m\}$ we have $d(X,Y,Z) = 1$, unless $X,Y,Z \in \{A_1,\dots,A_m\}$ or $X,Y,Z \in \{B_1,\dots,B_m\}$, in which case $d(X,Y,Z) = 0$ (recall the statement of Item (b) of Lemma \ref{main-lemma}). Therefore, if we pick $a_i$ (resp. $b_i$) vertices from $A_i$ (resp. $B_i$) for each $1 \leq i \leq m$ and denote by $X$ the set of all of these vertices, then  
    $$
	e(G[X]) = (a_1 + \dots + a_m) \sum_{1 \leq i < j \leq m} b_i b_j + (b_1 + \dots + b_m) \sum_{1 \leq i < j \leq m} a_i a_j.
	$$
    By Lemma \ref{additional-lemma}, the right-hand side takes $\Theta(m^3)$ values on the set of vectors \linebreak  $(a_1,\dots,a_m,b_1,\dots,b_m)$ with $a_i,b_i \geq 0$ and $\sum_{i=1}^m(a_i+b_i) = m$. Again, this implies the statement of Theorem \ref{main-result-2}, completing the proof.
\end{proof}

\subsection{Proof of Lemmas \ref{corollary-1} and \ref{additional-lemma}}\label{sec:cubic forms}


To prove Lemma \ref{corollary-1}, it is more convenient to rewrite the cubic polynomial in a different form. This is done in the following lemma, which easily implies Lemma \ref{corollary-1}.

\begin{lemma} \label{general-values-lemma}
    Let $A,B,C,D,E \in \mathbb{Q}$ such that $B \neq 0$ or $A,C \neq 0$.
    Then the function 
	$$
	f(x_1,\dots,x_m) = (Am+D) \sum_{1 \leq i \leq m} x_i^2 + 
	B \sum_{1 \leq i \leq m} x_i^3 + C \sum_{1 \leq i < j \leq m} x_i x_j (x_i - x_j)+E
	$$
    takes $\Omega(m^3)$ distinct values on the set of integer vectors $(x_1,\dots,x_m)$ with $x_1,\dots,x_m \geq 0$ and 
    $\sum_{i=1}^m x_i = m$.
\end{lemma}

Note that if $B=C=0$ then $f(x_1,\dots,x_m) = (Am+D)\sum_{i=1}^m x_i^2 + E$, so $f$ gets at most $m^2$ different values on 
vectors $(x_1,\dots,x_m)$ with $x_1,\dots,x_m \geq 0$ and $x_1 + \dots + x_m = m$, because $\sum_{i=1}^m x_i^2 \leq m^2$. 
This shows that the condition $B \neq 0$ or $C \neq 0$ in Lemma \ref{general-values-lemma} is necessary. 
The condition that $B \neq 0$ or $A \neq 0$ is not necessary; a modification of our argument works for this case as well. 
However, since this case never appears in our applications, we decided to avoid it.

\begin{proof}[Proof of Lemma \ref{general-values-lemma}]
    We may and will assume that $m$ is large enough as a function of $A,B,C,D$. 
    We first consider ``symmetric" vectors $x = (x_1,\dots,x_t)$, i.e., vectors satisfying $x_i = x_{t+1-i}$ for every $1 \leq i \leq t$. 
    To keep the notation cleaner, we ignore (i.e., do not write) coordinates $i \in [m]$ with $x_i = 0$, as these coordinates do not contribute to $f$. Namely, when we write $(x_1,\dots,x_t)$, we actually mean $(x_1,\dots,x_m)$ with $x_i = 0$ for $t < i \leq m$. 
    If $x = (x_1,\dots,x_t)$ is symmetric then 
    $$\sum_{1 \leq i < j \leq t} x_i x_j (x_i - x_j) = 
    \sum_{1 \leq i < j \leq t} x_i^2x_j - \sum_{1 \leq i < j \leq t} x_ix_j^2 = 0,$$ because $x_i^2x_j = x_{t+1-j}x_{t+1-i}^2$. This means that for symmetric vectors $x$, it holds that
    $$
    f(x) = (Am+ \nolinebreak D) \sum_{i} x_i^2 + 
	B \sum_{i} x_i^3 + E.
    $$

    Fix a constant $0 < \epsilon < \frac{1}{16}$ to be chosen later.
    For an integer $\epsilon m \leq \ell \leq 2\epsilon m$, let $x^{(\ell)}$ be the symmetric vector
    $$x^{(\ell)} = (\ell, \lfloor \sqrt{m} \rfloor, \dots, \lfloor \sqrt{m} \rfloor, 2, \dots, 2, 1, \dots, 1, 2, \dots, 2, \lfloor \sqrt{m} \rfloor, \dots, \lfloor \sqrt{m} \rfloor, \ell),$$ 
    where the number of coordinates equal to $\lfloor \sqrt{m} \rfloor$ is exactly $2 \lfloor \frac{\sqrt{m}}{8} \rfloor$; the number of $2$s is exactly $2 \lfloor \frac{m}{16} \rfloor$; and the number of $1$s is chosen so that the sum of all coordinates is exactly $m$. 
    We have 
    $$
    2\ell+ \lfloor \sqrt{m} \rfloor \cdot 
    2 \left\lfloor \frac{\sqrt{m}}{8} \right\rfloor + 
    2 \cdot 2 \left\lfloor \frac{m}{16} \right\rfloor \leq 2\ell + \frac{m}{2} \leq 4\epsilon m + \frac{m}{2} \leq \frac{3m}{4},
    $$
    so the number of $1$s is at least $\frac{m}{4}$. 
    

    By the assumptions of the lemma, $A,B$ cannot both be $0$. 
    For convenience, we will assume from now on that $A \geq 0$, and if $A = 0$ then $B > 0$ (this can be guaranteed by multiplying $f$ by $-1$, if necessary).
    We claim that $f(x^{(\ell+1)}) - f(x^{(\ell)}) = \Omega(m^2)$. 
    Indeed, recall that $x^{(\ell)}$ and $x^{(\ell+1)}$ have the same number of coordinates which equal $2$ and the same number of coordinates which equal $\lfloor \sqrt{m} \rfloor$. Also, the term $\sum_{1 \leq i < j \leq m} x_i x_j (x_i - x_j)$ in the definition of $f$ vanishes on both $x^{(\ell)}$ and $x^{(\ell+1)}$. Hence,
    \begin{align*}
        f(x^{(\ell+1)}) - f(x^{(\ell)})   
        &= 2(Am+D)\cdot ((\ell+1)^2 - \ell^2 - 1^2) + 2B \cdot ((\ell+1)^3 - \ell^3 - 1^3) \\
        &= 2\ell \cdot (2Am + 3B\ell + 2D + 3B).
    \end{align*}
    If $A,B \neq 0$ then choose $\epsilon$ to satisfy $\epsilon \leq \frac{A}{12|B|}$ (if $A$ or $B$ is 0 then there is no restriction).
    This ensures that if $A \neq 0$ then
    $|3B\ell| \leq 3|B| \cdot 2 \epsilon m \leq \frac{1}{2}Am$. Together with 
    $|2D + 3B| \leq \frac{1}{2}Am$ for $m$ large enough, we get that
    $2Am + 3B\ell + 2D + 3B \geq Am$ and thus 
    $f(x^{(\ell+1)}) - f(x^{(\ell)}) = 2\ell \cdot (2Am + 3B\ell + 2D + 3B) \geq 2 \epsilon m \cdot Am = \Omega(m^2)$, as required.
    If $A = 0$, then by assumption $B > 0$. Also, $|2D + 3B| \leq B \epsilon m \leq B\ell$ for $m$ large enough. Hence, $2\ell \cdot(2Am + 3B\ell + 2D + 3B) \geq 2\ell \cdot 2B\ell \geq
    2 \epsilon m \cdot 2B \epsilon m = 4B \epsilon^2 m^2 = \Omega(m^2)$.
    This proves our claim that $f(x^{(\ell+1)}) - f(x^{(\ell)}) = \Omega(m^2)$.

    So far, we showed that there exist $\Omega(m)$ vectors $x^{(\ell)}$, $\epsilon m \leq \ell \leq 2\epsilon m$, such that the difference in the $f$-values of two consecutive vectors is 
    $\Omega(m^2)$. 
	For the rest of the proof, we fix any $\epsilon m \leq \ell \leq 2\epsilon m$ 
    and perform certain operations on $x^{(\ell)}$ to get vectors $x = (x_1,\dots,x_m)$ with 
    $\Omega(m^2)$ different $f$-values which all lie between $f(x^{(\ell)})$ and $f(x^{(\ell+1)})$ or between $f(x^{(\ell-1)})$ and $f(x^{(\ell)})$. Doing this for every $\ell$ gives us $\Omega(m^3)$ different $f$-values, establishing the lemma. We consider several cases.

    \paragraph{Case 1:} $C \neq 0$.  
    In this case, starting with the vector $x^{(\ell)}$, we repeatedly apply the following operation: take a coordinate $i$ with $x_i = 1, x_{i+1} = 2$, and change this to $x_i = 2, x_{i+1} = 1$ (note that the resulting vectors are no longer symmetric). 
    One such operation changes the value of $f$ by 
    $C(2 \cdot 1(2-1) - 1 \cdot 2(1-2)) = 4C$. 
    Also, it is possible to perform this operation as long as there is a coordinate $i$ with $x_i = 1, x_{i+1} = 2$. Initially (i.e., in $x^{(\ell)}$), there are at least $\frac{m}{4}$ $1$s which come directly before $\lfloor \frac{m}{16} \rfloor$ $2$s, and we can perform the operation until we reach the situation where these $2$s come before the $1$s. Hence, we can perform this operation at least 
    $\frac{m}{4} \lfloor \frac{m}{16} \rfloor = \Omega(m^2)$ times. Therefore, there are $\Omega(m^2)$ different values of $f$, with two consecutive values at distance exactly $4C = \Theta(1)$. This means that there are $\Omega(m^2)$ different $f$-values between $f(x^{(\ell)})$ and $f(x^{(\ell+1)})$ or between 
    $f(x^{(\ell-1)})$ and $f(x^{(\ell)})$ (depending on the sign of $C$), as required. 
    
    \paragraph{Case 2:} $C = 0$. 
    So $f(x) = (Am+ \nolinebreak D) \sum_{i} x_i^2 + 
	B \sum_{i} x_i^3 + E$.
    Also, by assumption, $B \neq 0$. 
    In this case we will use two operations:
    \begin{itemize}
        \item An operation which changes the value of $f$ by some fixed number $p = \Theta(1)$.
        \item An operation which changes the value of $f$ by some number $q$ with $\alpha m \leq q \leq \beta m$, for some fixed absolute constants $\alpha,\beta > 0$.
    \end{itemize}
    Moreover, we will be able to perform each of these operations $\Omega(m)$ times independently of each other (again starting from $x^{(\ell)}$). 
    This will imply that there are $\Omega(m^2)$ different values of $f$ between $f(x^{(\ell)})$ and $f(x^{(\ell+1)})$, 
    because by applying the second operation $\Omega(m)$ times we get $\Omega(m)$ $f$-values which are $\Theta(m)$ apart, and we can then ``fill in these gaps" by applying the first operation $\Omega(m)$ times.
    
    The first operation is as follows: Choose distinct coordinates $i,j,k,h$ with $x_i = x_j = x_k = 2, x_h = 0$, and change this to $x_i = x_j = x_k = 1, x_h = 3$. This changes $f$ by 
    $$(Am+D)(3 \cdot 1^2 + 3^2 - 3 \cdot 2^2) + B(3 \cdot 1^3 + 3^3 - 3 \cdot 2^3) = 6B =: p = \Theta(1).$$ 
    Recall that $x^{(\ell)}$ has exactly $2 \cdot \lfloor \frac{m}{16} \rfloor$ coordinates which equal $2$. Since the sum of all coordinates is exactly $m$, and the total number of coordinates is $m$, there are also at least $2 \cdot \lfloor \frac{m}{16} \rfloor$ coordinates which equal $0$. 
    Therefore, it is possible to perform this operation at least 
    $\lfloor \frac{2}{3} \cdot \lfloor \frac{m}{16} \rfloor \rfloor = \Omega(m)$ times. 
    
    It remains to describe the second operation. Here we distinguish between two cases: 
    \paragraph{Case 2.1:} $A \neq 0$, so $A > 0$ (by our assumption). In this case we use the following operation: 
    Take coordinates $i,j$ with $x_i = x_j = 1$, and change this to $x_i = 2, x_j = 0$. This changes $f$ by 
    $
    q := (Am+D)(2^2 - 2 \cdot 1^2) + B(2^3 - 2 \cdot 1^3) = 2(Am+D) + 6B.
    $
    For large enough $m$ we have $Am \leq q \leq 3Am$, so $q = \Theta(m)$. 
    Since $x^{(\ell)}$ has at least $\frac{m}{4}$ $1$s, it is possible to perform this operation at least $\lfloor \frac{m}{8} \rfloor = \Omega(m)$ times, regardless of the number of times that the first operation was performed. 
    
    \paragraph{Case 2.2:} $A = 0$, so $B > 0$.
    In this case we use the following operation: take coordinates $i,j$ with  
    $\lfloor \sqrt{m} \rfloor \leq x_i \leq 2 \lfloor \sqrt{m} \rfloor$ and $x_j = 1$, increase $x_i$ by $1$, and set $x_j = 0$. This changes $f$ by 
    $$q := D((x_i+1)^2 - x_i^2 - 1^2) + B((x_i+1)^3 - x_i^3 - 1^3) = 
    3Bx_i^2 + (2D+3B)x_i.
    $$
    As $x_i = \Theta(\sqrt{m})$, we have $q = \Theta(m)$. 
    Recall that $x^{(\ell)}$ has exactly $2 \lfloor \frac{\sqrt{m}}{8} \rfloor$ coordinates which equal $\lfloor \sqrt{m} \rfloor$. For each such coordinate $i$, it is possible to perform this operation $\lfloor \sqrt{m} \rfloor$ times on $x_i$ (until $x_i$ becomes larger than $2 \lfloor \sqrt{m} \rfloor$). Also, $x^{(\ell)}$ has at least $\frac{m}{4}$ $1$s, meaning that there are enough $1$s for all of these $\Omega(m) \leq 2 \lfloor \frac{\sqrt{m}}{8} \rfloor \cdot \lfloor \sqrt{m} \rfloor \leq \frac{m}{4}$ operations. 
    So we can perform this second operation $\Omega(m)$ times, regardless of how many times we used the first operation.
    This completes the proof of the lemma. 
    %
\end{proof}


\begin{proof}[Proof of Lemma \ref{corollary-1}]
    Note that if $x_1 + \dots + x_m = m$ then 
    $$m^3 = (x_1 + \dots + x_m)^3 = \sum_{1 \leq i \leq m} x_i^3 + 3 \sum_{1 \leq i \neq j \leq m} x_i^2 x_j + 6 \sum_{1 \leq i < j < k \leq m} x_i x_j x_k.$$ 
    So we find:
    \begin{align*}
        &a \sum_{1 \leq i < j \leq m} x_i x_j^2 + b \sum_{1 \leq i < j \leq m} x_i^2 x_j + c \sum_{1 \leq i < j < k \leq m} x_i x_j x_k\\ 
        &= \frac{a+b}{2} \sum_{1 \leq i \neq j \leq m} x_i^2 x_j + \frac{b-a}{2} \sum_{1 \leq i < j \leq m} x_i x_j (x_i-x_j) + \frac{c}{6}\left( m^3 - \sum_{1 \leq i \leq m} x_i^3 - 3 \sum_{1 \leq i \neq j \leq m} x_i^2 x_j \right) \\
        &= \frac{a+b-c}{2} \sum_{1 \leq i \leq m} x_i^2 (m-x_i) + \frac{b-a}{2} \sum_{1 \leq i < j \leq m} x_i x_j (x_i-x_j) + \frac{c}{6} m^3 - \frac{c}{6} \sum_{1 \leq i \leq m} x_i^3,
    \end{align*}
    where the last equality uses 
    $\sum_{1 \leq i \neq j \leq m} x_i^2 x_j = \sum_{i=1}^m x_i^2(m-x_i)$, since $\sum_{i=1}^m x_i = m$.
    Also,
    $$\sum_{1 \leq i < j \leq m} x_i x_j = \frac{1}{2}(x_1 + \dots + x_m)^2 - \frac{1}{2}\sum_{i=1}^m x_i^2 = 
    \frac{1}{2}m^2 - \frac{1}{2}\sum_{i=1}^m x_i^2.$$
    It follows that
    \begin{align*}
        f(x_1,\dots,x_m) &= 
    \left( \frac{a+b-c}{2}m + d - \frac{e}{2} \right) \sum_{i=1}^m x_i^2 + \left( \frac{c}{3} - \frac{a+b}{2} \right)\sum_{i=1}^m x_i^3 \\ &+ 
    \frac{b-a}{2} \sum_{1 \leq i < j \leq m} x_i x_j (x_i-x_j) + \frac{c}{6}m^3 + \frac{e}{2}m^2
    \end{align*}
    So set $A=\frac{a+b-c}{2}, B = \frac{c}{3} - \frac{a+b}{2}, C = \frac{b-a}{2}, D = d - \frac{e}{2}, E = \frac{c}{6} m^3 + \frac{e}{2} m^2$ and apply Lemma \ref{general-values-lemma}. We only need to check that $B \neq 0$ or $A,C \neq 0$. If $B = C = 0$ then $a=b=\frac{c}{3}$, which is excluded by our assumption. And if $B = A = 0$ then $c = 0$ and $a = -b$, again excluded by our assumption.  
    Thus, we can apply Lemma \ref{general-values-lemma} to get the result. 
\end{proof}

\begin{proof}[Proof of Lemma \ref{additional-lemma}]
    Let $a,b \in \mathbb{N}$ be such that $a+b = m$ and $b-a \in \{1,2\}$ (depending on the parity of $m$). 
    We will only consider vectors $(a_1,\dots,a_m,b_1,\dots,b_m)$ where $a = a_1 + \dots + a_m$ and $b = b_1 + \dots + b_m$ and neglect writing out the coordinates where $a_i = 0$ or $b_i = 0$ (which do not contribute to $f$). Note that for such vectors, we have
    \begin{align*}
        f(a_1,\dots,a_m,b_1,\dots,b_m) &= 
        \frac{1}{2} \left( a \sum_{1 \leq i \neq j \leq m} b_i b_j + b \sum_{1 \leq i \neq j \leq m} a_i a_j \right) \\&= \frac{1}{2} \left( a b^2 + b a^2 - a \sum_{1 \leq i \leq m} b_i^2 - b \sum_{1 \leq i \leq m} a_i^2 \right).
    \end{align*}
    Thus, it suffices to show that the function
    $$
    g(a_1,\dots,a_m,b_1,\dots,b_m) := a \sum_{1 \leq i \leq m} b_i^2 + b \sum_{1 \leq i \leq m} a_i^2
    $$
    attains $\Omega(m^3)$ different values.
    The proof is similar to that of Lemma \ref{general-values-lemma}.
    First, for each integer $\frac{m}{8} \leq \ell \leq \frac{m}{4}$, let $x^{(\ell)}$ be the vector with $b_i = 1$ for each $1 \leq i \leq b$, $a_1 = \ell$, and $a_i = 1$ for each $2 \leq i \leq a-\ell+1$ (all other coordinates are $0$). 
    Then
    \begin{align*}
        g(x^{(\ell+1)}) - g(x^{(\ell)}) &= b((\ell+1)^2 - \ell^2 - 1^2) 
        = 2b\ell \geq 2 \cdot \frac{m}{2} \cdot \frac{m}{8} = \frac{m^2}{8}.
    \end{align*}
    So we have $\Omega(m)$ vectors $x^{(\ell)}$, $\frac{m}{8} \leq \ell \leq \frac{m}{4}$, such that the difference in $g$-values of two consecutive vectors is at least 
    $\frac{m^2}{8}$.

	For the rest of the proof, we fix a specific $\frac{m}{8} \leq \ell \leq \frac{m}{4}$ and apply operations to the vector $x^{(\ell)}$ to get many distinct values of $g$ between  $g(x^{(\ell)})$ and $g(x^{(\ell+1)})$. 
    We use the following two operations:
    \begin{itemize}
        \item Take coordinates $1 \leq i < j \leq m$ with $b_i = b_j = 1$ and change this to $b_i = 2, b_j = 0$. This changes $g$ by $a(2^2-1^2-1^2) = 2a = \Theta(m)$. 
        \item Take coordinates $1 \leq i < j \leq m$ and $1 \leq k < h \leq m$ with 
        $b_i = 2, b_j = 0, a_k = a_h = 1$, and change this to $a_k = 2, a_h = 0, b_i = b_j = 1$. This changes $g$ by 
        $a \cdot (1^2+1^2-2^2) + b \cdot (2^2-1^2-1^2) = 2(b-a) \in \{2,4\}$.  
    \end{itemize}
    In $x^{(\ell)}$ there are $b > \frac{m}{2}$ coordinates $i$ with $b_i=1$. Hence, the first operation can be applied at least $\lfloor \frac{m}{4} \rfloor$ times. Also, if we performed the first operation $r$ times, then there are $r$ disjoint pairs of indices $i,j$ with $b_i = 2, b_j = 0$. Moreover, $x^{(\ell)}$ has 
    $a-\ell \geq \lfloor \frac{m-2}{4} \rfloor$ indices $k$ with $a_k=1$, so we can perform the second operation 
    $\min\{r,\lfloor \frac{m-2}{8}\rfloor\}$ times. 
    By performing the first operation $r$ times and the second operation $s$ times, for $0 \leq s \leq r \leq \frac{m}{16}$, we change $g$ by 
    $$
    r \cdot 2a + s \cdot 2(b-a) <
    \frac{m}{16} \cdot m + \frac{m}{16} \cdot 4 \leq \frac{m^2}{8} \leq g(x^{(\ell+1)}) - g(x^{(\ell)}).
    $$
    Also,
    $s \cdot 2(b-a) \leq 4s \leq \frac{m}{4} < 2a$, 
    meaning that applying the second operation for any number of times $s \leq \frac{m}{16}$, we do not change $g$ by more than one application of the first operation. So we see that all $\binom{\lfloor m/16 \rfloor}{2} = \Omega(m^2)$ values of $g$ that we get this way are distinct and between $g(x^{(\ell)})$ and $g(x^{(\ell+1)})$. This implies that $g$ attains $\Omega(m^3)$ values in total, as required. 
\end{proof}

\appendix

\section{A counterexample to an exact version of Conjecture \ref{conj:many sizes}}
Here we describe a construction of J. Fox showing that for $r \geq 4$, there exist $r$-graphs with no polynomial-size homogeneous sets and with $s(G;m) \leq g_r(m)$, disproving a conjecture from \cite{ABGMW}. 
	
	We will need the main result of \cite{MR}, which states the following: For every $r \geq 4$ and every tournament $R$ whose edges are colored with $\binom{r}{2}$ distinct colors, every $m$-vertex $\binom{r}{2}$-edge-colored tournament $T$ contains at most $g_r(m)$ copies of $R$. In particular, let us take $R$ to be the transitive tournament with vertices $1,\dots,r$ (with $i \rightarrow j$ for all $1 \leq i < j \leq r$), and color the edge $ij$ with color $c_{i,j}$. Now, take a transitive tournament $T$ on $n$ vertices $v_1,\dots,v_n$, with $v_i \rightarrow v_j$ for $1 \leq i < j \leq r$, and color each edge randomly with one of the $\binom{r}{2}$ colors $c_{i,j}$, $1 \leq i < j \leq r$. Let $G_r$ be the $r$-graph whose edges correspond to copies of $R$ in $T$. It is a standard fact that with high probability, $G_r$ only has homogeneous sets of size $O(\log n)$. Moreover, by the aforementioned result of \cite{MR}, if $r \geq 4$, then every set of $m$ vertices in $G_r$ spans at most $g_r(m)$ edges. This immediately shows that $s(G_r;m) \leq g_r(m) + 1$. 
	We state this fact for later reference:
	\begin{fact}\label{fact:MV}
		For every $m \geq r \geq 4$, $G_r$ has no $m$ vertices with more than $g_r(m)$ edges. 
	\end{fact}
	Note that for $m=2r$ we have $g_r(2r) = 2^r$.
	To improve the bound to $s(G_r;2r) \leq g_r(2r)$ (for $r \geq 4$), we show that there are no $m = 2r$ vertices spanning exactly $2^r-1$ edges\footnote{We note that this is not true for $r=3$, i.e., $G_3$ can have $2^3-1=7$ edges on $6$ vertices. Indeed, consider vertices $x_1<x_2<\dots<x_6$ such that all edges between $\{x_1,x_2,x_3\}$ and $\{x_4,x_5\}$ have color $c_{1,2}$; all edges between $\{x_1,x_2,x_3\}$ and 
    $x_6$ have color $c_{1,3}$; all edges between $\{x_4,x_5\}$ and $x_6$ have color $c_{2,3}$; and $x_ix_j$ has color $c_{i,j}$ for $1 \leq i < j \leq 3$.}. 
    We will only give the proof for $r \geq 5$. The case $r=4$ requires a lengthier case analysis, which we omit. 
    So suppose by contradiction that $x_1,\dots,x_{2r}$ span exactly $2^r-1$ edges. 
	For convenience, let $H := G[\{x_1,\dots,x_{2r}\}]$ denote the subgraph of $G$ induced on $\{x_1,\dots,x_{2r}\}$. So $e(H) = 2^r-1$. First we prove the following:
	\begin{claim}\label{claim:x1,x2}
		No edge of $H$ contains both $x_1$ and $x_2$.
	\end{claim}
	\begin{proof}
		Suppose otherwise, and let $e \in E(H)$ be an edge with $x_1,x_2 \in e$. Let $d^+(x_2)$ denote the number of edges of $H$ in which $x_2$ is the first element; i.e., this is the number of edges of $H$ which contain $x_2$ but not $x_1$. For every $x_i \in e \setminus \{x_1,x_2\}$, the color of $x_2x_i \in E(T)$ is $c_{2,\ell}$ for some $3 \leq \ell \leq r$, because $x_2$ is the second vertex of $e$. This means that there is no edge which contains $x_i$ and where $x_2$ is the first vertex, because for such an edge to exist, the color of $x_2x_i$ would have to be $c_{1k}$ for some $k$. Hence, for every edge $f \in E(H)$ in which $x_2$ is the first vertex, it holds that $f \setminus \{x_2\} \subseteq \{x_3,\dots,x_{2r}\} \setminus e =: A$. We have $|A| = |\{x_3,\dots,x_{2r}\} \setminus e| = (2r-2) - (r-2) = r$. Now, note that $d^+(x_2)$ is precisely the number of $(r-1)$-tuples $y_1 < \dots < y_{r-1}$ in $A$ such that the color of $y_iy_j \in E(T)$ is $c_{i+1,j+1}$ for every $1 \leq i < j \leq r-1$. By Fact \ref{fact:MV}, $A$ has at most $g_r(r-1) = 2$ such $(r-1)$-tuples. Hence, $d^+(x_2) \leq 2$. 
		
		Next, consider the edges of $H$ on $\{x_3,\dots,x_{2r}\}$. By Fact \ref{fact:MV}, the number of such edges is at most $g_r(2r-2) = 2^{r-2}$. Letting $d(x_1)$ denote the degree of $x_1$ in $H$, note that 
		$2^r-1 = e(H) = d(x_1) + d^+(x_2) + e(\{x_3,\dots,x_{2r}\}) \leq d(x_1) + 2 + 2^{r-2}$. Hence, $d(x_1) \geq 3 \cdot 2^{r-2} - 3$. 
		
		Now we consider the edges of $H$ containing $x_1$. For each $2 \leq i \leq r$, let $X_i$ be the set of vertices $x \in \{x_2,\dots,x_{2r}\}$ such that the edge $x_1x \in E(T)$ has color $c_{1,i}$. Then $d(x_1) \leq \prod_{i=2}^r |X_i|$. As $d(x_1) \geq 3 \cdot 2^{r-2} - 3 > 2^{r-1}$ (using $r \geq 4$), it must be that $\sum_{i=2}^r |X_i| = 2r-1$ (otherwise $\prod_{i=2}^r |X_i| \leq 2^{r-1}$). Moreover, if $r \geq 5$ then the multiset $\{|X_2|,|X_3|,\dots,|X_r|\}$ must be $\{2,2,\dots,2,3\}$ (with $r-2$ times $2$). Indeed, if not, then there is $2 \leq j \leq r$ with $|X_j|=1$. But then $\prod_{i \in [2,r] \setminus \{j\}}|X_i| \leq 3^2 \cdot 2^{r-4}$ (because the product is maximized when the sets $(X_i)_{i \in [2,r] \setminus \{j\}}$ are as equal as possible, which is when two of the sets have size $3$ and the remaining $r-4$ have size $2$). But $9 \cdot 2^{r-4} < 3 \cdot 2^{r-2} - 3$ for $r \geq 5$, a contradiction. 
		So indeed $\{|X_2|,|X_3|,\dots,|X_r|\} = \{2,\dots,2,3\}$ as a multiset, as claimed. 
        
        Now, we have $d(x_1) \leq \prod_{i=2}^r |X_i| = 3 \cdot 2^{r-2} < 2^r-1$.
		Hence, there must be an edge $f$ of $H$ inside $\{x_2,\dots,x_{2r}\}$. 
        Write $f = \{y_1 < \dots < y_r\}$. Since $r \geq 4$ and $|X_i| \leq 3$ for every $i$, at least one of the edges $y_1y_{\ell}$ ($2 \leq \ell \leq r$) is not contained in any of the sets $X_i$; i.e., it goes between two different sets $X_i,X_j$ ($2 \leq i < j \leq r$). Since the color of $y_1y_{\ell}$ is $c_{1,\ell} \neq c_{i+1,j+1}$, the edge $y_1y_{\ell}$ cannot be contained in an edge of $H$ containing $v_1$. The number of $(r-1)$-tuples in $X_2 \times X_3 \times \dots \times X_r$ containing $y_1y_{\ell}$ is 
		$\prod_{k \in [2,r] \setminus \{i,j\}}|X_k| \geq 2^{r-3}$.
        This means that 
        $d(x_1) \leq \prod_{k=2}^r |X_k| - 2^{r-3} = 3 \cdot 2^{r-2} - 2^{r-3} < 3 \cdot 2^{r-2} -3$ (using $r\geq 5$), a contradiction. This proves Claim \ref{claim:x1,x2}.
	\end{proof}
    \noindent
    We now consider the degrees of $x_1,x_2$ in $H$. We will use $d(x)$ to denote the degree of $x$ in $H$. 
    \begin{claim}\label{claim:degrees of x1,x2}
        For $z \in \{x_1,x_2\}$, if $d(z) \geq 2^{r-1}-1$ then $d(z) = 2^{r-1}$, and $H$ has no edges inside 
        $\{x_3,\dots,x_{2r}\}$.
    \end{claim}
    \begin{proof}
        By Claim \ref{claim:x1,x2}, every edge containing $z$ is contained in $\{z\} \cup \{x_3,\dots,x_{2r}\}$. 
        For each $2 \leq i \leq r$, let $X_i$ be the set of $x \in \{x_3,\dots,x_{2r}\}$ such that the color of $zx \in E(T)$ is $c_{1,i}$. Then $2^{r-1}-1 \leq d(z) \leq \prod_{i=2}^r |X_i|$, which is only possible if $|X_2| = \dots = |X_r| = 2$, as $\sum_{i=2}^r |X_i| \leq |\{x_3,\dots,x_{2r}\}| = 2r-2$. Moreover, for every $2 \leq i < j \leq r$, if there is an edge in $E(X_i,X_j)$ which is not colored by $c_{i+1,j+1}$, then this edge cannot participate in any edge of $H$ containing $z$, which means that $d(z) \leq 2^{r-1} - 2^{r-3} < 2^{r-1}-1$, a contradiction. So all edges between $X_i$ and $X_j$ are colored with $c_{i+1,j+1}$. 
        In particular, every $(r-1)$-tuple in $X_2 \times X_3 \times \dots \times X_r$ makes an edge with $z$, so $d(z) = 2^{r-1}$. 
        It also follows that $\{x_3,\dots,x_{2r}\}$ cannot contain any edges of $H$. Indeed, assuming that $f = \{y_1 < \dots < y_r\}$ is an edge of $H$ inside $\{x_3,\dots,x_{2r}\}$, there must be $2 \leq \ell \leq r$ such that the edge $y_1y_{\ell}$ goes between two of the sets $X_2,\dots,X_r$, say between $X_i$ and $X_j$. But then $y_1y_{\ell}$ has color $c_{i+1,j+1} \neq c_{1,\ell}$, in contradiction to $f$ being an edge of $H$.
    \end{proof}
    
    By Fact \ref{fact:MV}, the number of edges of $H$ on $\{x_2,\dots,x_{2r}\}$ is at most $g_r(2r-1) = 2^{r-1}$. Hence, $d(x_1) \geq e(H) - 2^{r-1} = 2^r - 1 - 2^{r-1} = 2^{r-1}-1$. Now, by Claim \ref{claim:degrees of x1,x2}, $d(x_1) = 2^{r-1}$, and all edges of $H$ touch $x_1$ or $x_2$. Hence, 
    $d(x_2) \geq e(H)-d(x_1) = 2^r-1 - d(x_1) = 2^{r-1}-1$. So by Claim \ref{claim:degrees of x1,x2}, $d(x_2) = 2^{r-1}$. But then $e(H) = d(x_1) + d(x_2) = 2^{r-1} + 2^{r-1} = 2^r$, a contradiction.

\end{document}